\documentclass[12pt]{amsart}
\usepackage{pgf,tikz,pgfplots} 
\pgfplotsset{compat=1.14}
\usepackage{mathrsfs}
\usetikzlibrary{arrows}
\usetikzlibrary{patterns}
\usepackage{pgfplots}
\usetikzlibrary{intersections, pgfplots.fillbetween}
\usepackage[colorlinks=true,urlcolor=blue, citecolor=green,linkcolor=blue,linktocpage,pdfpagelabels, bookmarksnumbered,bookmarksopen]{hyperref}
\usepackage[hyperpageref]{backref}
\usepackage{cleveref}
\usepackage{xcolor}
\usepackage{amsthm} 
\usepackage{latexsym,amsmath,amssymb}
\usepackage{accents}
\usepackage{a4wide}
\usepackage{soul}
\usepackage{mathtools} 
\usepackage{xparse} 
\usepackage{mathtools}

\pgfdeclarelayer{ft}
\pgfdeclarelayer{bg}
\pgfsetlayers{bg,main,ft}

\title{\protect{A fractional version of Rivi\`ere's {GL(n)}-gauge}}

\author{Francesca Da Lio}
\address[Francesca Da Lio]{Department of Mathematics, ETH Z\"urich, R\"amistrasse 101, 8092 Z\"urich, Switzerland}
\email{francesca.dalio@math.ethz.ch}

\author{Katarzyna Mazowiecka}
\address[Katarzyna Mazowiecka]{\change{
Lehrstuhl f\"{u}r Angewandte Analysis, RWTH Aachen University, Pontdriesch 14-16, 52062 Aachen, germany}}
\email{\change{mazowiecka@math1.rwth-aachen.de }}

\author{Armin Schikorra}
\address[Armin Schikorra]{Department of Mathematics,
University of Pittsburgh,
301 Thackeray Hall,
Pittsburgh, PA 15260, USA}
\email{armin@pitt.edu}

plus 6pt minus 12pt
\def\eps{\varepsilon}

\def\vp{\varphi}

\def\N{{\mathbb N}}

\newcommand{\dif}{\,\mathrm{d}}
\newcommand{\dx}{\dif x}
\newcommand{\dy}{\dif y}

\newtheorem{theorem}{Theorem}

\newtheorem{lemma}[theorem]{Lemma}
\newtheorem{corollary}[theorem]{Corollary}
\newtheorem{proposition}[theorem]{Proposition}

\theoremstyle{definition}

\def\dist{{\rm dist\,}}

\def\lip{{\rm Lip\,}}

\def\supp{{\rm supp\,}}

\newcommand{\R}{\mathbb{R}}

\newcommand{\brac}[1]{\left (#1 \right )}
\newcommand{\abs}[1]{\left |#1 \right |}
\newcommand{\Ep}{\bigwedge\nolimits}

\newcommand{\barint}{
\rule[.036in]{.12in}{.009in}\kern-.16in \displaystyle\int }

\newcommand{\barcal}{\mbox{$ \rule[.036in]{.11in}{.007in}\kern-.128in\int $}}

\def\mvint_#1{\mathchoice
          {\mathop{\vrule width 6pt height 3 pt depth -2.5pt
                  \kern -8pt \intop}\nolimits_{\kern -3pt #1}}%
          {\mathop{\vrule width 5pt height 3 pt depth -2.6pt
                  \kern -6pt \intop}\nolimits_{#1}}%
          {\mathop{\vrule width 5pt height 3 pt depth -2.6pt
                  \kern -6pt \intop}\nolimits_{#1}}%
          {\mathop{\vrule width 5pt height 3 pt depth -2.6pt
                  \kern -6pt \intop}\nolimits_{#1}}}

\numberwithin{theorem}{section} \numberwithin{equation}{section}

\renewcommand{\div}{\operatorname{div}}

\newcommand{\lap}{\Delta }
\newcommand{\aleq}{\precsim}
\newcommand{\ageq}{\succsim}
\newcommand{\aeq}{\approx}

\newcommand{\laps}[1]{(-\lap)^{\frac{#1}{2}}}
\newcommand{\lapv}{(-\lap)^{\frac{1}{4}}}
\newcommand{\laph}{(-\lap)^{\frac{1}{2}}}

\definecolor{chameleongreen1}{RGB}{98,189,25}
\newcommand{\change}[1]{{#1}}

\begin{document}
\begin{abstract}
We prove that for antisymmetric vectorfield $\Omega$ with small $L^2$-norm there exists a gauge $A \in L^\infty \cap \dot{W}^{1/2,2}(\R^1,GL(N))$ such that 
\[
 \div_{\frac12} (A\Omega- d_{\frac{1}{2}} A) = 0.
\]
This extends a celebrated theorem by Rivi\`ere   to the nonlocal case and provides conservation laws for a class of nonlocal equations with antisymmetric potentials, as well as stability under weak convergence. 
\end{abstract}

\sloppy
\sloppy
\maketitle

\tableofcontents
\section{Introduction}
In the celebrated work \cite{R07} Rivi\`ere showed that for two-dimensional disks $D \subset \R^2$ for any $\Omega \in L^2(D,so(N)\otimes \Ep^1 \R^2)$, i.e., $\Omega_{ij}=-\Omega_{ji} \in L^2(D,\Ep^1 \R^2)$ there exists a $GL(N)$-gauge, namely a matrix-valued function $A,A^{-1} \in L^\infty \cap W^{1,2}(D,GL(N))$ such that 
\[
\div(A\Omega - \nabla A) = 0.
\]
These are distortions of the orthonormal Uhlenbeck's Coulomb gauges, \cite{U82}, namely $P  \in L^\infty\cap W^{1,2}(D,SO(N))$ which satisfy
\[
 \div(P\Omega P^t - P\change{^t}\nabla P) = 0.
\]
As Rivi\`ere showed in \cite{R07}, the $GL(N)$-gauges have the advantage that they can transform equations of the form
\begin{equation}\label{eq:rivsys}
 -\lap u = \Omega \cdot \nabla u
\end{equation}
into a conservation law 
\[
 \div(A \nabla u)=\div((\nabla A - A\Omega)u).
\]
This is important since \eqref{eq:rivsys} is the structure of the equation for harmonic maps, $H$-surfaces, and more generally the Euler-Lagrange equations of a large class of conformally invariant variational functionals. The $GL(N)$-gauge transform allows for regularity theory and the study of weak convergence \cite{R07}, it also is an important tool for energy quantization, see \cite{LR14}.

In recent years a theory of fractional harmonic maps has developed, beginning with the work by Rivi\`ere and the first named author, \cite{DLR11,DLR11b}. 
 bubbling analysis was initiated in \cite{DL15}. Fractional harmonic maps have a variety of applications: they appear as free boundary of minimal surfaces or harmonic maps \cite{M11,MS15,S18,DLP20}, they are also related to nonlocal minimal surfaces \cite{MSW19} and to knot energies \cite{BRS16,BRS19}.\par
 
\change{We recall that  in \cite{DLR11} the first named author and Rivi\`ere considered nonlocal Sch\"odinger type systems of the form
  \begin{equation}\label{eqabstr}
 (-\Delta)^{\frac{1}{4}} v=\Omega v\quad\quad\mbox{in  ${\mathcal{D}}'(\R),$} \
 \end{equation}
where $\Omega$ is an  antisymmetric potential in $L^2(\R,so(N))$, $v\in  L^2(\R,{\R}^N)$. The main technique to  establish the sub-criticality of systems \eqref{eqabstr}   is to perform  a {\it change of gauge} by rewriting them
 after having multiplied $v$ by a well chosen rotation valued map $P\in \dot{W}^{1/2,2}({\R},SO(N))$ which is ''integrating'' $\Omega$ in an optimal
 way. 
The key point in \cite{DLR11,DLR11b} was  the discovery of   particular algebraic structures (three-term commutators) that play the role
 of the Jacobians in the case of local systems in $2$-D with an antisymmetric potential   and that enjoy suitable integrability by compensations properties.
In \cite{MS18} the second and the third named authors introduced a new approach to fractional harmonic maps by 
considering  nonlocal systems with an antisymmetric potential which is seen itself as a nonlocal operator.  As we will explain later such an approach is similar in the spirit to that introduced by H\'elein  in \cite{H90} in the context of harmonic maps.}

It  begins with the definition of ``nonlocal one forms''. $F \in L^p(\Ep_{od}^1 \R^n)$ if $F: \R^n \times \R^n \to \R$ and 
\[
 \int_{\R^n}\int_{\R^n} |F(x,y)|^p\, \frac{\dx \dy}{|x-y|^n} <\infty.
\]
The $s$-differential, which takes function $u: \R^n \to \R$ into 1-forms, is then given by
\[
 d_s u(x,y) \coloneqq \frac{u(x)-u(y)}{|x-y|^s}.
\]
The scalar product for two $1$-forms, $F \in L^p(\Ep_{od}^1 \R^n)$ and $G \in L^{p'}(\Ep_{od}^1 \R^n)$ is then given by
\[
 F\cdot G(x) = \int_{\R^n} F(x,y) G(x,y) \frac{dy}{|x-y|^n}.
\]
The fractional divergence $\div_s$, which takes 1-forms into functions, is then the formal adjoint to $d_s$, namely
\[
 \div_s F[\varphi] \coloneqq \int_{\R^n} F \cdot d_s \varphi \quad \forall \varphi \in C_c^\infty(\R^n).
\]
For more details we refer to \Cref{s:preliminaries}. With this notation in mind we now consider equations of the form
\begin{equation}\label{eq:antisympde}
 \div_{\frac{1}{2}} (d_{\frac{1}{2}} u) = \Omega \cdot d_{\frac{1}{2}} u\quad \text{in $\R$} ,
\end{equation}
or in index form
\[
 \div_{\frac{1}{2}} (d_{\frac{1}{2}} u^i) = \sum_{j=1}^N \Omega_{ij} \cdot d_{\frac{1}{2}} u^j \quad \text{in $\R$,\quad $i=1,\ldots,N$},
\]
where $u \in (L^2 + L^\infty)\cap \dot W^{\frac{1}{2},2}(\R,\R^N)$ and $\Omega_{ij} = - \Omega_{ji} \in L^2(\Ep^1_{od} \R)$.

The main observation in \cite{MS18} is that the above notation and the above equation are not merely some random definitions of only analytical interest. Rather it was shown that the role of \eqref{eq:antisympde} for fractional harmonic maps is similar to the role of \eqref{eq:rivsys} for harmonic maps. In \cite{MS18} it was shown that there exists a $div-curl$ Lemma in the spirit of \cite{CLMS93}, that fractional harmonic maps into spheres satisfy a conservation law in the spirit of \cite{H90}, and that fractional harmonic maps into spheres essentially satisfy equations of the form \eqref{eq:antisympde}, in the spirit of \cite{R07}, and that an analogue of Uhlenbeck's gauge exist. In \cite{MPS20} this argument was further pushed to equations of stationary harmonic map in higher dimensional domains.

\change{We mention that in \cite{DLR21} the authors found quasi conservation laws  for  nonlocal Schr\"odinger type systems  of the form
\begin{equation}\label{modelsystemold}
(-\Delta)^{1/4} v=\Omega v+g(x)\end{equation}
where $v\in L^2(\R)$, $\Omega\in L^2(\R,so(N))$, and $g$ is a tempered distribution.
As we have already pointed out above  systems \eqref{modelsystemold}
 represent a particular case of systems \eqref{eq:antisympde} studied in the present paper in the sense that the  antisymmetric potential $\Omega$ in
 \eqref{modelsystemold} is    a pointwise function.
The  conservation laws found  in  \cite{DLR21} are a consequence of a 
  stability property of some three-term commutators by the multiplication of $P \in SO(N)$ and also of  the regularity results obtained
previously for such commutators.
The reformulation of \eqref{modelsystemold} in terms of conservation laws has permitted 
 to get the  quantization  in the neck regions of the $L^2$ norms of the negative part of sequences of solutions to    systems of the type \eqref{modelsystemold}.}

\change{
 The conservation laws that we obtain  in the current paper are more similar in the spirit to those found in the paper \cite{R07} for harmonic maps and concern nonlocal systems \eqref{eq:antisympde} where the antisymmetric potential acts in general as a nonlocal operator.
We hope this technique to be as useful for the   question of concentration compactness and  energy quantization for systems
 as it was in the local case in \cite{LR14}, a question we will study in a future work.}

Applying a gauge $A \in L^\infty \cap \dot{W}^{\frac{1}{2},2}$ to the equation \eqref{eq:antisympde} we find (see \Cref{le:reformulateOmegatoA}),
 \[
\div_{\frac{1}{2}} (A_{ik} d_{\frac{1}{2}} u^k) =\brac{A_{i\ell}\Omega_{\ell k}\, -d_{\frac{1}{2}} A_{ik} } \cdot d_{\frac{1}{2}}u^k.
  \]
Our main result is then the existence of the nonlocal analogue of Rivi\`ere's $GL(N)$-Coulomb gauge \cite{R07}, namely we have

\begin{theorem}\label{th:main}
 There exists a number $0<\sigma \ll 1$ such that the following holds.
 
 If $\Omega\in L^2(\Ep_{od}^1\R)$ is antisymmetric, i.e., $\Omega_{ij} = -\Omega_{ji}$ and satisfies
 \[
  \|\Omega\|_{L^2(\Ep_{od}^1 \R)} < \sigma,
 \]
then there exists an invertible matrix valued function $A \in L^\infty\cap \dot{W}^{\frac12,2}(\R,GL(N))$ such that for $\Omega^A \coloneqq A\Omega - d_\frac12 A$ we have
\[
  \div_\frac12 \brac{\Omega^A} =0.
\]
Moreover we have
\begin{equation}\label{eq:Aestimates}
[A]_{W^{\frac12,2}(\R)} \aleq \|\Omega\|_{L^2(\Ep_{od}^1 \R)}, \quad \|A\|_{L^\infty(\R)}\aleq 1+ \|\Omega\|_{L^2(\Ep_{od}^1 \R)}.
\end{equation}
\end{theorem}

As an immediate corollary we obtain 

\begin{corollary}[Conservation law]\label{co:conservation}
Assume \change{$u\in \dot{W}^{\frac12,2}(\R,\R^N)\cap (L^2+L^\infty)(\R,\R^N)$ and $f\in \dot{W}^{-\frac12,2}(\R,\R^N)$ satisfy}
\[
 \div_{\frac{1}{2}} (d_{\frac{1}{2}} \change{u}) = \Omega \cdot d_{\frac{1}{2}} u+f, ~~\mbox{in ${\mathcal{D}}'(\R)$}
\]
and $\Omega$ satisfies the condition of \Cref{th:main}. Then there exists a matrix $A$ such that for $\Omega^A \coloneqq A\Omega - d_\frac12 A$ we have
\[
 \div_{\frac{1}{2}} \brac{A d_{\frac{1}{2}} u- \change{(\Omega^A)^\ast} u}=Af,  \quad~~\text{ in } {\mathcal{D}}'(\R),
\]
where $\change{(\Omega^A)^\ast(x,y)\coloneqq \Omega^A(y,x)}$.
\end{corollary}

\Cref{th:main} is applicable to the half-harmonic map system as derived \cite[Proposition 4.2]{MS18}, because of a localization result, see \Cref{pr:localization}.

With the methods of \Cref{th:main} we obtain the analogue of \cite[Theorem I.5]{R07}, our second main result.
\begin{theorem}\label{th:weakconvergence} 
Assume $\Omega_\ell\in L^2(\Ep_{od}^1 \R)$ is a sequence of antisymmetric vector fields, i.e., $(\Omega_{ij})_\ell = - (\Omega_{ji})_\ell$, weakly convergent in $L^2$ to an $\Omega \in L^2(\Ep_{od}^1 \R)$. Assume further that $f_\ell \in \dot{W}^{-\frac12,2}(\R, \R^N)$ converges strongly to $f$ in $\dot{W}^{-\frac12,2}$, and assume that $u_\ell\in (L^2+L^\infty(\R))\cap\dot{W}^{\frac12,2}(\R,\R^N)$ is a sequence of solutions to
\begin{equation}\label{eq:uellequation}
 (-\Delta)^{\frac12} u_\ell = \Omega_{\ell} \cdot d_\frac12 u_\ell + f_\ell \quad~~\mbox{in~ ${\mathcal{D}}'(\R)$}
\end{equation}
such that $\sup_{\ell}\brac{\|u_\ell\|_{L^2+L^\infty(\R)}+[u_\ell]_{W^{\frac12,2}(\R)}}<\infty$.
Then, up to taking a subsequence $u_\ell$ converges weakly in $\dot{W}^{\frac12,2}(\R,\R^N)$ to some $u\in \dot{W}^{\frac12,2}(\R,\R^N) \cap ((L^2 + L^\infty)(\R,\R^N))$, which is a solution to
\[
 (-\Delta)^\frac12 u = \Omega \cdot d_\frac12 u + f   ~~\mbox{in ~${\mathcal{D}}'(\R).$}
\]
\end{theorem}
Here, as usual, we denote \[\|f\|_{L^2+L^\infty(\R)} = \inf_{f_1 \in L^2(\R)} \brac{\|f_1\|_{L^2(\R)} + \|f-f_1\|_{L^\infty(\R)}}.\]
\Cref{th:weakconvergence} will be proven in \Cref{s:weakconv}.

\subsection*{Acknowledgment}
Funding is acknowledged as follows 
\begin{itemize}
 \item (FDL) Swiss National Fund, SNF200020\_192062: Variational Analysis in Geometry;
 \item (KM) FSR Incoming Post-doctoral Fellowship;
 \item (AS) Simons Foundation (579261).
\end{itemize}
\change{The authors would also like to thank the anonymous referee for helpful comments.}

\section{Preliminaries and useful tools}\label{s:preliminaries}
We follow the notation of \cite{MS18} for the nonlocal operators. For readers convenience we recall it here. We write \change{$\mathcal M(\R^n)$ for the space of all functions $f\colon \R^n \to \R$ measurable with respect to the Lebesgue measure $\dx$ and} $\mathcal M(\Ep_{od}^1\R^n)$ for the space of vector fields $F\colon \R^n \times \R^n \to \R$ measurable with respect to the $\frac{\dx \dy}{|x-y|^n}$ measure, where $``od"$ stands for ``off diagonal''.

For two vector fields $F,\, G \in \mathcal M(\Ep_{od}^1 \R^n)$ the scalar product is defined as
\[
F \cdot G (x) \coloneqq \int_{\R^n} F(x,y) \, G(x,y) \frac{\dif y}{|x-y|^n}.
\]

For any $p>1$ the natural $L^p$-space on vector fields $F\colon \R^n\times \R^n \to \R$ is induced by the norm 
\[
	\|F\|_{L^p(\Ep_{od}^1 \R^n)} \coloneqq \brac{\int_{\R^n} \int_{\R^n} |F(x,y)|^p\frac{\dif x \dy}{|x-y|^n}}^\frac1p
\]
and for $D\subset \R^n$ we define
\[
 \|F\|_{L^p(\Ep_{od}^1 D)} \coloneqq \brac{\iint_{(D\times \R^n)\cup(\R^n\times D)} |F(x,y)|^p \frac{\dif x \dif y}{|x-y|^n}}^\frac1p.
\]
\change{Let $s\in(0,1)$. }For $f\colon \R^n \to \R$ we let the $s$-gradient $d_s \colon \mathcal M(\R^n)\to \mathcal M(\Ep_{od}^1 \R^n)$ to be
\[
d_s f(x,y) \coloneqq \frac{f(x) - f(y)}{|x-y|^s}.
\]
Observe that with this notation we have
\[
 \|d_s f\|_{L^p(\Ep_{od}^1 \R^n)} = [f]_{W^{s,p}(\R^n)},
\]
where 
\[ [f]_{W^{s,p}(\R^n)}=\left(\int_{\R^n}\int_{\R^n}\frac{|f(x)-f(y)|^p}{|x-y|^{n+sp}}\dx\dy\right)^{1/p}\]
is the Gagliardo--Slobodeckij seminorm.

\change{Let $s\in(0.1)$ and $F\in \mathcal M(\Ep_{od}^1 \R^n)$.} We define the fractional $s$-divergence in the distributional way
\[
\div_{s} F[\vp] \coloneqq \int_{\R^n} \int_{\R^n} F(x,y)\, d_s \vp(x,y) \frac{\dx \dy}{|x-y|^n}, \quad \vp\in C_c^\infty(\R^n),
\]
\change{whenever the integrals converge.}

With this notation we have $\div_s d_s = (-\Delta)^s$, i.e.,
\[
 \int_{\R^n} d_s f \cdot d_s g(x) \dif x= \change{\frac{2}{C_{n,s}}} \int_{\R} (-\Delta)^s f(x) g(x) \dif x, 
\]
where the fractional Laplacian is defined as
\[
 (-\Delta)^s f (x) \coloneqq \change{C_{n,s}} P.V. \int_{\R^n} \frac{f(x)- f(y)}{|x-y|^{2s}}\frac{\dif y}{|x-y|^n}. 
\]

A simple observation is the following
\begin{lemma}\label{la:prodrule}
Let $F \in \change{\mathcal{M}}(\Ep^1_{od} \R^n)$ then we define 
\[
 F^\ast(x,y) \coloneqq \change{F(y,x)}.
\]
If $\div_s F = 0$ then $\div_s F^\ast = 0$.

Moreover, for any $F \in \change{\mathcal{M}}(\Ep^1_{od} \R^n)$ and $u\in \change{\mathcal{M}}(\R^n)$ we have 
\change{
\begin{equation}\label{eq:productrulewithx}
 \div_{s} (Fu(x)) = \div_s(F) u+F^\ast \cdot d_{s}u
 \end{equation}
and 
 \begin{equation}\label{eq:productrulewithy}
 \div_{s} (Fu(y)) = \div_s(F) u-F \cdot d_{s}u
 \end{equation}
 }
\change{whenever each term is well-defined.}
\end{lemma}
\begin{proof}
 We have 
\[
 F(x,y) u(x) (\varphi(x)-\varphi(y)) = F(x,y)  (u(x)\varphi(x)-u(y)\varphi(y)) - F(x,y) (u(x)-u(y))\varphi(y).
\]
Thus,
\begin{equation}\label{eq:productrule1}
\begin{split}
 \change{\div_s( Fu(x))[\vp] }&=\int_{\R^n} \int_{\R^n} \frac{F(x,y) u(x) (\varphi(x)-\varphi(y))}{|x-y|^{n+s}}\, \dy \dx\\
 &= 
 \int_{\R^n} \int_{\R^n} \frac{F(x,y)  (u(x)\varphi(x)-u(y)\varphi(y))}{|x-y|^{n+s}}\, \dy \dx\\
 &\quad
 - \int_{\R^n} \int_{\R^n} \frac{F(x,y)  (u(x)-u(y))\varphi(y)
 }{|x-y|^{n+s}}\, \dy \dx.
 \end{split}
\end{equation}
As for the latter term we have
\begin{equation}\label{eq:productrule2}
\begin{split}
 - &\int_{\R^n} \int_{\R^n} \frac{F(x,y)  (u(x)-u(y))\varphi(y)}{|x-y|^{n+s}} \dy \dx\\
 &=\change{-}\int_{\R^n} \int_{\R^n} \frac{-F(y,x)  (u(x)-u(y))\varphi(x)}{|x-y|^{n+s}} \dy \dx\\
 &= \int_{\R^n} \int_{\R^n} \frac{F^\ast (x,y)  (u(x)-u(y))\varphi(x)}{|x-y|^{n+s}} \dy \dx.
\end{split}
\end{equation}
\change{Combining \eqref{eq:productrule1} with \eqref{eq:productrule2} we obtain \eqref{eq:productrulewithx}. The proof of \eqref{eq:productrulewithy} is similar.}
\end{proof}

We also denote
\[
|D_{s,q} f|(x) \coloneqq \brac{\int_{\R^n}\frac{|f(x) - f(y)|^q}{|x-y|^{n+sq}} \dy}^\frac1q.
\]
We will be using the following ``Sobolev embedding'' theorem. 
\begin{theorem}\label{th:weirdsobolev}
	Let $s\in(0,1)$, $t\in(s,1)$, and let $p,\,p^*>1$ satisfy
	\[
	s-\frac{n}{p^*} = t - \frac{n}{p},
	\]
	where $q>1$ with $p^*>\frac{nq}{n+sq}$. Then we have
	\begin{equation}\label{eq:weirdsobolev}
	\||\mathcal D_{s,q}f|\|_{L^{p^*}(\R^n)} \aleq \|\laps{t} f\|_{L^p(\R^n)}
	\end{equation}
	and for any $r\in[1,\infty]$
	\begin{equation}\label{eq:weridsobolevlorentz}
	\||\mathcal D_{s,q}f|\|_{L^{(p^*,r)}(\R^n)} \aleq \|\laps{t} f\|_{L^{(p,r)}(\R^n)}.
	\end{equation}
\end{theorem}
For the proof see \Cref{s:weirdsob}.

We will also need the following Wente's inequality from \cite{MS18}.
\begin{lemma}[{\cite[Corollary 2.3]{MS18}}]\label{la:Wentewithestimate}
 Let $s\in(0,1)$, $p>1$, and let $p'$ be the H\"{o}lder conjugate of $p$. Assume moreover that $F'\in L^p(\Ep_{od}^1 \R)$ and $g\in W^{s,p'}(\R)$ with $\div_s F=0$.
Let $R$ be a linear operator such that for some $\Lambda>0$ satisfies
 
 \[
  |R[\vp]| \le \Lambda \|(-\Delta)^{\frac 14} \vp\|_{L^{(2,\infty)}(\R)},
 \]
 where $L^{(2,\infty)}(\R)$ denote the weak $L^2$ space.
 Then any distributional solution $u\in \dot{W}^{\frac12,2}(\R)$ to
\[
 \laph u = F \cdot d_s g + R   \quad \text{in } \R
\]
is continuous. Moreover if $\lim_{x \to \pm \infty} |u(x)|=0$, then we have the estimate
\begin{equation}\label{eq:nlocwente}
\|u\|_{L^\infty(\R)}+\|d_\frac12 u\|_{L^2(\Ep_{od}^1\R)} \aleq \|F\|_{L^p(\Ep_{od}^1 \R)} \|d_s g\|_{L^{p'}(\Ep_{od}^1 \R)} + \Lambda.
\end{equation}

\end{lemma}

Our proof will also be based on the following choice of a good gauge.

\begin{theorem}[{\cite[Theorem 4.4]{MS18}}]\label{th:gauge}
For $\Omega_{ij} = -\Omega_{ji} \in L^2(\Ep^1_{od}\R)$ there exists $P \in \dot{W}^{\frac{1}{2}}(\R,SO(N))$ such that
\[
 \div_{\frac{1}{2}} \Omega^P_{ij} = 0 \quad \mbox{for all }i,j \in \{1,\ldots,N\},
\]
where
\[
 \Omega^P = \frac{1}{2} \brac{d_{\frac{1}{2}}P(x,y) \brac{P^T(y) + P^T(x)} -  P(x) \Omega(x,y) P^T(y) - P(y) \Omega(x,y) P^T(x)}
\]
  and
  \begin{equation}\label{eq:Pestimateinte}
   [P]_{W^{\frac12,2}(\R)} \aleq \|\Omega\|_{L^2(\Ep_{od}^1 \R)}.
  \end{equation}
\end{theorem}

\section{Proof of Theorem~\ref{th:main}}
In this section we prove \Cref{th:main}.   We will be looking for an $A$ in the form $A=(I+\eps)P$, where $P$ is chosen to be the good gauge from \Cref{th:gauge}. 
The idea to take perturbation of rotations of the form $(I+\eps)P$ has been taken from \cite{R12} in the context of local Schr\"odinger equations with antisymmetric potentials. This has been also  exploited in \cite{DLR21}. 

\begin{lemma}\label{la:rewrite1}
Assume that $A = (I+\eps) P$. 

Then for
\[
 \Omega^P(x,y) = \frac{1}{2} \brac{d_{\frac{1}{2}}P(x,y) \brac{P^T(y) + P^T(x)} -  P(x) \Omega(x,y) P^T(y) - P(y) \Omega(x,y) P^T(x)}
\]
we have
\[
A(x)\Omega(x,y)\, -d_{\frac{1}{2}} A(x,y)
=- (I+\eps(x))\, \Omega^P(x,y)\, P(y)
- d_{\frac{1}{2}} \eps(x,y)\, P(y)
+R_\eps(x,y),
\]
where $R_\eps$ is given by the formula
\begin{equation}\label{eq:Restdefinition}
\begin{split}
R_\eps(x,y) \coloneqq \frac{1}{2} (I+\eps(x)) \bigg( &d_{\frac{1}{4}} P(x,y)\, d_{\frac{1}{4}} P^T(x,y) \\
&-P(x)\, \Omega(x,y)\, \brac{P^T(x)-P^T(y)}\\
&+\brac{P(x)-P(y)}\,\Omega(x,y)\, P^T(x)
\bigg ) P(y).
\end{split}
\end{equation}
\end{lemma}
\begin{proof}
Recall that
\[
 d_\frac12(fg)(x,y) = d_\frac12 f(x,y)\, g(y) + f(x) d_\frac12 g(x,y). 
\]
Thus, applying this to $d_\frac12 ((I+\eps)P)(x,y)$ we get
\begin{equation}\label{eq:rewriteAOdiv}
\begin{split}
A(x)&\Omega(x,y) - d_\frac12 A(x,y)\\
&=(I+\eps(x)) P(x)\Omega(x,y)\, -d_{\frac{1}{2}} \brac{(I+\eps) P}(x,y)\\
&=(I+\eps(x)) \brac{P(x)\, \Omega(x,y)\, -d_{\frac{1}{2}} P(x,y)} -d_{\frac{1}{2}} \eps(x,y)\, P(y) \\
&=-(I+\eps(x)) \brac{d_{\frac{1}{2}} P(x,y)\, P^T(y)-P(x)\, \Omega(x,y)\, P^T(y)} P(y) -d_{\frac{1}{2}} \eps(x,y)\, P(y).
\end{split}
\end{equation}
Next we observe that
\begin{equation}\label{eq:rewriteplugOmegaP}
\begin{split}
&d_{\frac{1}{2}} P(x,y)\, P^T(y)-P(x)\, \Omega(x,y)\, P^T(y)\\
&=\frac{1}{2}\brac{d_{\frac{1}{2}} P(x,y)\, \brac{P^T(x)+P^T(y)}-P(x)\, \Omega(x,y)\, P^T(y)-P(y)\,\Omega(x,y)\, P^T(x)}\\
&\quad -\frac{1}{2} \Big(d_{\frac{1}{2}} P(x,y)\, \brac{P^T(x)-P^T(y)} -P(x)\, \Omega(x,y)\, \brac{P^T(x)-P^T(y)}\\
&\phantom{\quad -\frac{1}{2} \quad}+\brac{P(x)-P(y)}\,\Omega(x,y)\, P^T(x)\Big).
\end{split}
\end{equation}

That is, plugging in \eqref{eq:rewriteplugOmegaP} into \eqref{eq:rewriteAOdiv} we get the claim for 
\begin{equation*}
\begin{split}
 R_\eps(x,y) \coloneqq \frac{1}{2} (I+\eps(x)) \bigg( &d_{\frac{1}{4}} P(x,y)\, d_{\frac{1}{4}} P^T(x,y) \\
&-P(x)\, \Omega(x,y)\, \brac{P^T(x)-P^T(y)}\\
&+\brac{P(x)-P(y)}\,\Omega(x,y)\, P^T(x)
\bigg ) P(y).
\end{split}
\end{equation*}
\end{proof}

\begin{lemma}\label{le:aisvanishing}
	Assume that we have $\eps\in L^\infty\cap \dot{W}^{1/2,2}(\R),\, a\in \dot{W}^{1/2,2}(\R)$, and $B\in L^2(\Ep_{od}^1 \R)$ satisfying the equations
	\begin{equation}\label{eq:nlocwitha}
	- (I+\eps(x))\, \Omega^P(x,y)\, P(y)
	- d_{\frac{1}{2}} \eps(x,y)\, P(y)
	+R_\eps(x,y) = d_\frac12 a(x,y) + B(x,y)
	\end{equation}
	and
	\begin{equation}\label{eq:nlocwithPwithouta}
	\begin{split}
	- &\div_{\frac{1}{2}}\brac{(I+\eps(x))\, \Omega^P(x,y)}
	- \div_{\frac{1}{2}}\brac{d_{\frac{1}{2}} \eps(x,y)}
	+\div_{\frac{1}{2}}(R_\eps(x,y) P^T(y))\\
	&= \div_{\frac{1}{2}} \brac{B(x,y) P^T(y)},
	\end{split}
	\end{equation}
	with
	\begin{equation}\label{eq:assumptiononP}
	[P]_{W^{1/2,2}(\R)} < \sigma.
	\end{equation}
	Then, for sufficiently small $\sigma$ we have $a=const$.
\end{lemma}

\begin{proof}
	We multiply \eqref{eq:nlocwitha} by $P^T(y)$ from the right and take the $\frac12$-divergence on both sides, then subtracting \eqref{eq:nlocwithPwithouta} we obtain
	\begin{equation}\label{eq:nlocdivgradientaP}
	\div_\frac12 (d_\frac12 a(x,y)P^T(y)) = 0.
	\end{equation}
	We use nonlocal Hodge decompostion \Cref{la:hodge} and get the existence of functions $\tilde a \in \dot{W}^{\frac12,2}(\R)$, $\tilde B\in L^2(\Ep_{od}^1 \R)$ \change{such that}
	\begin{equation}\label{eq:nlonhodgefora}
	d_\frac12 a(x,y)P^T(y) = d_\frac12 \tilde a (x,y) + \tilde B(x,y),
	\end{equation}
	and (recall $|P|=1$)
	\begin{equation}\label{eq:nlochodgeforaBtildeestimate}
	\div_\frac12 \tilde B =0 \quad \text{and } \quad \|\tilde B\|_{L^2(\Ep_{od}^1 \R)}  \aleq \|d_\frac12 a\|_{L^2(\Ep_{od}^1 \R)}.
	\end{equation}
	 Thus, taking the $\frac12$-divergence in \eqref{eq:nlonhodgefora} we obtain
	\[
	0 = \div_\frac12 (d_\frac12 a(x,y)P^T(y)) = \div_\frac12 (d_\frac12 \tilde a (x,y) + \tilde B(x,y)) = \div_\frac12 (d_\frac12 \tilde a) = \laps{1} \tilde a. 
	\]
	This gives, $\laps{1}\tilde a =0$, \change{thus $\tilde a$ is constant and without loss of generality we can take} $\tilde a = 0$\change{, see also \cite[Theorem 1.1]{Fall-2016}}. Thus \eqref{eq:nlonhodgefora} becomes
	\[
	d_\frac12 a(x,y)P^T(y) = \tilde B(x,y).
	\]
	That is
	\[
	d_\frac12 a(x,y) = \tilde B(x,y) P(y).
	\]
	Taking the $\frac12$-divergence we obtain by \Cref{la:prodrule}
	\begin{equation}\label{eq:}
	\laps{1} a = \change{-\tilde B} \cdot d_\frac12 P, 
	\end{equation}
	since on the righ-hand side we have a \emph{div-curl} term we can apply fractional Wente's inequality, \Cref{la:Wentewithestimate}, and obtain from \eqref{eq:nlocwente}
	\[
	 \|d_\frac 12 a\|_{L^2(\Ep_{od}^1 \R)} \aleq \|\tilde B\|_{L^2(\Ep_{od}^1 \R)} \|d_\frac12 P\|_{L^2(\Ep_{od}^1 \R)}.
\]
Combining this with \eqref{eq:nlochodgeforaBtildeestimate} and \eqref{eq:assumptiononP} we get
\[
 \|d_\frac 12 a\|_{L^2(\Ep_{od}^1 \R)} \aleq \sigma \|d_\frac 12 a\|_{L^2(\Ep_{od}^1 \R)},
\]
which implies for sufficiently small $\sigma$ that
\[
 \|d_\frac 12 a\|_{L^2(\Ep_{od}^1 \R)} = [a]_{W^{1/2,2}(\R)}=0
\]
and thus $a \equiv const$.
\end{proof}
Now we will focus on showing that there exists a solution to the equations \eqref{eq:nlocwitha} and \eqref{eq:nlocwithPwithouta}. We will do this by using the Banach fixed point theorem.

\begin{proposition}\label{pr:nlocfixedpoint}
Let $\Omega\in L^2(\Ep_{od}^1 \R)$ be anitsymmetric. There is a number $0<\sigma\ll1$ such that the following holds:

Take $P\in \dot{W}^{\frac12,2}(\R,SO(N))$ and $\Omega^P\in L^2(\Ep_{od}^1\R)$ from \Cref{th:gauge}. Let us assume that
\begin{equation}\label{eq:smallnessPandOmega}
 [P]_{W^{1/2,2}(\R)} + \|\Omega\|_{L^2(\Ep_{od}^1 \R)} < \sigma.
\end{equation}
Then, there exist $\eps\in L^\infty\cap \dot{W}^{1/2,2}(\R)$, $a\in \dot{W}^{1/2,2}(\R)$, and $B\in L^2(\Ep_{od}^1 \R)$ that solve the equations
 	\begin{equation}
 	\left\{
 	\begin{array}{l}
	- (I+\eps(x))\, \Omega^P(x,y) P(y)
	\change{-} d_{\frac{1}{2}} \eps(x,y) P(y)
	+R_\eps(x,y) = d_\frac12 a(x,y) + B(x,y) \\
	- \div_{\frac{1}{2}}\brac{(I+\eps(x)) \Omega^P(x,y)}
	\change{-} \div_{\frac{1}{2}}(d_{\frac{1}{2}} \eps(x,y))
	+\div_{\frac{1}{2}}(R_\eps(x,y) P^T(y)) = \div_{\frac{1}{2}} \brac{B P^T(y)},
	\end{array}
	\right.
	\end{equation}
	where $R_\eps$ is defined in \eqref{eq:Restdefinition}.
	
	Moreover, $\eps$ satisfies the estimate
	\begin{equation}\label{eq:eps-estimates}
	 \|\eps\|_{L^\infty(\R)} + [\eps]_{W^{\frac12,2}(\R)} \aleq \|\Omega\|_{L^2(\Ep_{od}^1 \R)}.
	\end{equation}
\end{proposition}

We will need the following remainder terms estimates.
\begin{lemma}\label{le:Restimates}
We have the following estimates 
\begin{equation}\label{eq:restestimate1}
\begin{split}
 &\abs{\div_{\frac12} (R_\eps P^T\change{(y)})[\varphi]}\\[3mm]
 &\quad \quad \aleq (1+\|\eps\|_{L^\infty(\R)}) (\|\Omega\|_{L^2(\Ep_{od}^1\R)} + [P]_{W^{1/2,2}(\R)})[P]_{W^{1/2,2}(\R)}\, \|\lapv \varphi\|_{L^{(2,\infty)}(\R)}
\end{split}
 \end{equation}
and
\begin{equation}\label{eq:restestimatedifference}
 \begin{split}
 &\abs{\div_{\frac12} \brac{(R_{\eps_1} - R_{\eps_2}) P^T\change{(y)}}[\varphi]}\\
 &\quad \quad \aleq \|\eps_1-\eps_2\|_{L^\infty(\R)} (\|\Omega\|_{L^2(\Ep_{od}^1\R)} + [P]_{W^{1/2,2}(\R)})[P]_{W^{1/2,2}(\R)}\, \|\lapv \varphi\|_{L^{(2,\infty)}(\R)}.
\end{split}
 \end{equation}
\end{lemma}
\begin{proof}
We observe that for any $\varphi \in C_c^\infty(\R)$ we have 
 \begin{equation}\label{eq:splitintotwo}
 \begin{split}
  &\left |\div_{\frac{1}{2}}(R_\eps P^T\change{(y)})[\varphi] \right |\\
  &\aleq\abs{\int_\R \int_\R 
  (I+\eps(x)) \brac{d_{\frac{1}{4}} P(x,y)\, d_{\frac{1}{4}} P^T(x,y) 
} \ d_{\frac{1}{2}} \varphi(x,y) \frac{\dx \dy}{|x-y|}}\\
  &\quad +\abs{\int_\R \int_\R 
	(I+\eps(x)) \brac{P(x)\, \Omega(x,y)\, \brac{P^T(x)-P^T(y)}\ d_{\frac{1}{2}} \varphi(x,y)\frac{\dx \dy}{|x-y|}}}\\
	&\quad +\abs{\int_\R \int_\R 
		(I+\eps(x))\brac{P(x)-P(y)}\,\Omega(x,y)\, P^T(x) \ d_{\frac{1}{2}} \varphi(x,y)\frac{\dx \dy}{|x-y|}}\\
&\aleq \brac{1+\|\eps\|_{L^\infty}}\, \int_\R \int_\R \brac{|d_{\frac{1}{4}} P(x,y)|^2\, |d_{\frac{1}{2}} \varphi(x,y)|\, +|\Omega(x,y)|\, |d_{\frac{1}{4}}P(x,y)|\,
| d_{\frac{1}{4}} \varphi(x,y)|} \frac{\dx \dy}{|x-y|}\\
&=\brac{1+\|\eps\|_{L^\infty}}\, \brac{\mathcal{I} + \mathcal{II}}.
\end{split}
\end{equation}
Let $\mathcal M$ be the Hardy--Littlewood maximal function and let $\alpha\in(0,1)$. We will use the following fractional counterpart (for the proof see \cite[Proposition 6.6]{S18})
\begin{equation}\label{eq:maximalfractional}
 |f(x) - f(y)| \aleq |x-y|^\alpha \brac{ \mathcal M ((-\Delta)^{\frac{\alpha}{2}}f)(x) +   \mathcal M ((-\Delta)^{\frac{\alpha}{2}}f)(y)}
\end{equation}
of the well known inequality, see \cite{Bojarski-Hajlasz-1993, Hajlasz-1996}
\[
 |f(x) - f(y)| \aleq |x-y| \brac{\mathcal M|\nabla f|(x) + \mathcal M |\nabla f|(y)}.
\]
We begin with the estimate of the first term on the right-hand side of \eqref{eq:splitintotwo}. 

We observe that by \eqref{eq:maximalfractional} and by the symmetry of the integrals we obtain
\begin{equation}\label{eq:asdfg}
\begin{split}
\mathcal I\coloneqq \int_\R \int_\R |d_\frac14 P(x,y)|^2 |d_\frac12 \vp(x,y)|\frac{\dx \dy}{|x-y|}
\aleq \int_{\R} |\mathcal M (\lapv \vp)(x)| \int_\R |d_\frac14 P(x,y)|^2 \frac{\change{\dy \dx}}{|x-y|}.
\end{split}
\end{equation}
Applying H\"{o}lder's inequality (for Lorentz spaces) we obtain
\begin{equation}\label{eq:asdkjhfgasdj}
\begin{split}
\int_{\R} |\mathcal M (\lapv \vp)(x)| \int_\R |d_\frac14 P(x,y)|^2 \frac{\change{\dy \dx}}{|x-y|}
&\aleq \|\lapv \vp\|_{L^{(2,\infty)}} \||\mathcal D_{\frac14,2} P|^2\|_{L^{(2,1)}}\\
&= \|\lapv \vp\|_{L^{(2,\infty)}} \||\mathcal D_{\frac14,2} P|\|^2_{L^{(4,2)}},
\end{split}
\end{equation}
where we used the notation from \Cref{s:preliminaries}: for $s\in(0,1)$ and $q>1$ we write 
\[|\mathcal D_{s,q}f|(x) \coloneqq \brac{\int_\R \frac{|f(x)-f(y)|^q}{|x-y|^{1+sq}}  \dy}^\frac1q.\] 

Applying \Cref{th:weirdsobolev}, \eqref{eq:weridsobolevlorentz} for $t=\frac{1}{2}$ we get
\begin{equation}\label{eq:qwere}
\||\mathcal D_{\frac14,2} P|\|^2_{L^{(4,2)}} \aleq \| \lapv P\|^2_{L^{(2,2)}} \aleq \|\lapv P\|^2_{L^2} = [P]_{W^{1/2,2}}^2.
\end{equation}
Thus, combining \eqref{eq:asdfg}, \eqref{eq:asdkjhfgasdj}, and \eqref{eq:qwere} we obtain
\begin{equation}\label{eq:justrest1}
\mathcal I = \int_\R \int_\R|d_\frac13 P(x,y)|^2 |d_\frac13 \vp(x,y)|\frac{\dx \dy}{|x-y|} \aleq [P]_{W^{1/2,2}(\R)}^2 \|\lapv \vp\|_{L^{(2,\infty)}(\R)}.
\end{equation}
As for the second term of \eqref{eq:splitintotwo} we have
\begin{equation}\label{eq:second2}
\begin{split}
\mathcal{II}&\coloneqq \int_\R \int_\R |\Omega(x,y)||d_\frac14 P(x,y)||d_\frac14 \vp (x,y)|\frac{\dx \dy}{|x-y|}\\ 
&\aleq \|\Omega\|_{L^2(\Ep_{od}^1 \R)} \brac{\int_{\R} \int_{\R} |d_\frac14 P(x,y)|^2 |d_\frac14 \vp(x,y)|^2 \frac{\dif x \dif y}{|x-y|}}^\frac12.
\end{split}
\end{equation}
Applying once again \eqref{eq:maximalfractional} we obtain
\begin{equation}\label{eq:second3}
\begin{split}
\int_{\R} \int_{\R}& |d_\frac14 P(x,y)|^2 |d_\frac14 \vp(x,y)|^2 \frac{\dif x \dif y}{|x-y|} \\
&\aleq \int_\R \int_\R \brac{\mathcal M ((-\Delta)^{\frac18}\vp)(x) + \mathcal M ((-\Delta)^{\frac18}\vp) (y)}^2 |d_\frac14 P(x,y)|^2 \frac{\dif x \dif y}{|x-y|}\\
&\aleq \int_\R  \brac{\mathcal M ((-\Delta)^{\frac18}\vp)(x)}^2 \int_\R |d_\frac14 P(x,y)|^2 \frac{\change{\dif y \dif x }}{|x-y|}.
\end{split}
\end{equation}
Using H\"{o}lder's inequality and then Sobolev embedding we get
\begin{equation}\label{eq:second4}
\begin{split}
\int_\R  \brac{\mathcal M ((-\Delta)^{\frac18}\vp)(x)}^2 \int_\R |d_\frac14 &P(x,y)|^2 \frac{\change{ \dif y \dif x}}{|x-y|}\\
&\aleq \|(\mathcal M (-\Delta)^\frac18 \vp)^2 \|_{L^{(2,\infty)}(\R)} \| |\mathcal D_{\frac14,2} P|^2\|_{L^{(2,1)}(\R)}\\
&\aleq \|(-\Delta)^\frac18 \vp \|^2_{L^{(4,\infty)}(\R)} \| |\mathcal D_{\frac14,2} P|\|^2_{L^{(4,2)}(\R)}\\
&\aleq \|(-\Delta)^\frac14 \vp \|^2_{L^{(2,\infty)}(\R)} \| |\lapv P\|^2_{L^{(2,2)}(\R)},
\end{split}
\end{equation}
where for the estimate of the last term we used again \Cref{th:weirdsobolev}, \eqref{eq:weridsobolevlorentz}, with $t=\frac12$.

Combining \eqref{eq:second2}, \eqref{eq:second3}, and \eqref{eq:second4} we obtain
\begin{equation}\label{eq:justrest2}
\begin{split}
\mathcal{II} = \int_\R \int_\R |\Omega(x,y)||d_\frac14 P(x,y)|&|d_\frac14 \vp (x,y)|\frac{\dx \dy}{|x-y|}\\
&\aleq \|\Omega\|_{L^2(\Ep_{od}^1 \R)} \|(-\Delta)^\frac14 \vp \|_{L^{(2,\infty)}(\R)} [ P]_{W^{1/2,2}(\R)}.
\end{split}
\end{equation}
Finally, from \eqref{eq:splitintotwo}, \eqref{eq:justrest1}, and \eqref{eq:justrest2} we get
\[
\begin{split}
\Big|\div_\frac12&(R_\eps P^T\change{(y)})[\vp]\Big|\\
&\aleq  (1+ \|\eps\|_{L^\infty(\R)}) \brac{\|\Omega\|_{L^2(\Ep_{od}^1 \R)} + [P]_{W^{1/2,2}(\R)} }[P]_{W^{1/2,2}(\R)}\|\lapv \vp\|_{L^{(2,\infty)}(\R)}.
\end{split}
\]
This finishes the proof of \eqref{eq:restestimate1}.

In order to prove \eqref{eq:restestimatedifference} we observe
\[
 \abs{\div_{\frac12} \brac{(R_{\eps_1} - R_{\eps_2}) P^T\change{(y)}}[\varphi]} \aleq \|\eps_1 - \eps_2\|_{L^\infty}(\mathcal I + \mathcal{II}).
\]
Thus, in order to conclude it suffices to apply the estimates \eqref{eq:justrest1} and \eqref{eq:justrest2}.
\end{proof}

\begin{proof}[Proof of \Cref{pr:nlocfixedpoint}]
Let $X = L^\infty \cap \dot{W}^{\frac{1}{2},2}(\R)$. 

For any $\eps \in X$ we have $A=(1+\eps)P\in L^\infty\cap \dot{W}^{\frac12}(\R)$, which implies $A\Omega - d_\frac12 A \in L^2(\Ep_{od}^1 \R)$ and thus, from \Cref{la:rewrite1}, we have
\[
-\brac{(I+\eps(x,y))\, \Omega^P(x,y) P(y)}
\change{-} \brac{d_{\frac{1}{2}} \eps(x,y) P(y)}
+R_\eps(x,y)\in L^2(\Ep_{od}^1 \R). 
\]
We apply for this term the nonlocal Hodge decomposition, \Cref{la:hodge}: given $\eps\in X$ we
find $a(\eps)\in W^{\frac12,2}(\R)$ and $B(\eps)\in L^2(\Ep_{od}^1\R)$ with $\div_\frac12 B(\eps)=0$ satisfying
\begin{equation}\label{eq:hodgeeps}
\begin{split}
- &\brac{(I+\eps(x,y))\, \Omega^P(x,y) P(y)}
\change{-} \brac{d_{\frac{1}{2}} \eps(x,y) P(y)}
+R_\eps(x,y)\\
&= d_{\frac{1}{2}} a(\eps)(x,y) + B(\eps)(x,y)
\end{split}
\end{equation}
with the estimates
\begin{equation}\label{eq:Bestimate}
 \begin{split}
 &\|B(\eps)\|_{L^2(\Ep_{od}^1 \R)} + [a(\eps)]_{W^{\frac12,2}(\R)}\\
 &\quad\quad\aleq (1+\|\eps\|_{L^\infty(\R)})([P]_{W^{1/2,2}(\R)} + \|\Omega\|_{L^2(\Ep_{od}^1\R)}) + [\eps]_{W^{1/2,2}(\R)}.
\end{split}
 \end{equation}
Similarly, if for any two $\eps_1,\, \eps_2 \in X$ we consider the difference of the corresponding equations \eqref{eq:hodgeeps} we get
\begin{equation}\label{eq:Bdifferenceestimate}
\begin{split}
&\|B(\eps_1) - B(\eps_2)\|_{L^2(\Ep_{od}^1 \R)}\\
&\quad \quad\aleq \|\eps_1 - \eps_2\|_{L^\infty(\R)}([P]_{W^{1/2,2}(\R)} + \|\Omega\|_{L^2(\Ep_{od}^1\R)}) + [\eps_1 - \eps_2]_{W^{1/2,2}(\R)}.
\end{split}
\end{equation}

Now we define the mapping $T\colon X \to X$ as the solution to 
\begin{equation}\label{eq:Tdefinition}
\begin{split}
-& \div_{\frac{1}{2}}\brac{(I+\eps(x))\, \Omega^P(x,y)}
\change{-} \div_{\frac{1}{2}}\brac{d_{\frac{1}{2}} T(\eps)(x,y)}
+\div_{\frac{1}{2}}(R_\eps(x,y) P^T(y))\\
&= \div_{\frac{1}{2}} \brac{B(\eps)(x,y)P^T(y)}
\end{split}
\end{equation}
with $\lim_{|x|\to\infty}T(\eps)(x)=0$.

\change{Using \Cref{la:prodrule}} equation \eqref{eq:Tdefinition} can be rewritten as
\begin{equation}\label{eq:Tfractionallap}
\begin{split}
 \change{-}\laps{1} T(\eps) 
 &= \div_{\frac{1}{2}} \brac{B(\eps)P^T\change{(y)}} + \div_{\frac{1}{2}}\brac{(I+\eps\change{(x)})\, \Omega^P} - \div_{\frac{1}{2}}(R_\eps P^T\change{(y)})\\
 &= \change{-B(\eps)} \cdot d_\frac12 P^T +  d_\frac12(I+\eps)\cdot(\Omega^P)^\ast - \div_\frac12(R_\eps P^T\change{(y)})\change{.}
 \end{split}
 \end{equation}
We used in the second inequality \Cref{la:prodrule}.
 
We observe that on the right-hand we have fractional $div$-$curl$-terms: $\div_\frac12 \change{B(\eps)} =0$ and $\div_\frac12 (\Omega^P)^\ast =0$. Let us denote 
\[
 \Lambda_{\eps} \coloneqq (1+\|\eps\|_{L^\infty(\R)}) (\|\Omega\|_{L^2(\R)} + [P]_{W^{1/2,2}(\R)})[P]_{W^{1/2,2}(\R)}.
\]
By \Cref{le:Restimates}, \eqref{eq:restestimate1}, the rest term in \eqref{eq:Tfractionallap} satisfies 
\[
 \abs{\div_{\frac12} (R_\eps P^T\change{(y)})[\varphi]} \aleq \Lambda_\eps \|\lapv \varphi\|_{L^{(2,\infty)}(\R)}.
\]
Thus, we may apply the nonlocal Wente's lemma, i.e.,  \Cref{la:Wentewithestimate} and obtain
\begin{equation}\label{eq:Tepsestimate}
\begin{split}
&\|T(\eps)\|_{L^\infty(\R)} + [T(\eps)]_{W^{1/2,2}(\R)}\\
&\quad\aleq \|\change{B(\eps)}\|_{L^2(\Ep_{od}^1\R)}[P]_{W^{1/2,2}(\R)} + [\eps]_{W^{1/2,2}(\R)}\|(\Omega^P)^\ast\|_{L^2(\Ep_{od}^1 \R)}+ \Lambda_\eps\\
&\quad =\|B(\eps)\|_{L^2(\Ep_{od}^1\R)}[P]_{W^{1/2,2}(\R)} + [\eps]_{W^{1/2,2}(\R)}\|\Omega^P\|_{L^2(\Ep_{od}^1 \R)}+ \Lambda_\eps.
\end{split}
\end{equation}
Moreover, let $\eps_1,\, \eps_2 \in X$, then we have 
\begin{equation}\label{eq:Tfractionallapdifference}
\begin{split}
 &\change{-}\laps{1} (T(\eps_1)-T(\eps_2)) \\
 &= \div_{\frac{1}{2}} \brac{(B(\eps_1) - B(\eps_2))P^T(y)} +  \div_{\frac{1}{2}}\brac{(\eps_1-\eps_2)\change{(x)}\, \Omega^P} - \div_{\frac{1}{2}}((R_{\eps_1}-R_{\eps_2}) P^T\change{(y)})\\
 &=\change{-(B(\eps_1) - B(\eps_2))}\cdot d_\frac12 P^T +  d_{\frac{1}{2}}(\eps_1-\eps_2)\cdot (\Omega^P)^\ast - \div_{\frac{1}{2}}((R_{\eps_1}-R_{\eps_2}) P^T\change{(y)}),
\end{split}
 \end{equation}
 where in the last equality we have used again \Cref{la:prodrule}.
 
Again, we observe that 
\[
 \div_\frac12(\change{B(\eps_1) - (B(\eps_2)}) =0 \quad \text{and} \quad \div_\frac12 (\Omega^P)^\ast=0,
\]
and from \Cref{le:Restimates}, \eqref{eq:restestimatedifference}, we may estimate the reminder term in \eqref{eq:Tfractionallapdifference}
\begin{equation}\label{eq:Tdifferences}
 |\div_{\frac12} (R_{\eps_1} - R_{\eps_2}) P^T\change{(y)})[\varphi]| \aleq \Lambda_{\eps_1,\eps_2} \|\lapv \varphi\|_{L^{(2,\infty)}(\R)},
\end{equation}
where
\begin{equation}\label{eq:Lambdaeps1eps2}
 \Lambda_{\eps_1,\eps_2} \coloneqq \|\eps_1-\eps_2\|_{L^\infty(\R)} ([P]_{W^{1/2,2}(\R)} + \|\Omega\|_{L^2(\R)})[P]_{W^{1/2,2}(\R)}.
\end{equation}
Therefore, we may apply the nonlocal Wente's \Cref{la:Wentewithestimate} for equation \eqref{eq:Tfractionallapdifference} and obtain
\begin{equation}\label{eq:Tdifferenceestimate}
\begin{split}
 &\|T(\eps_1) - T(\eps_2)\|_{L^\infty(\R)} + [T(\eps_1) - T(\eps_2)]_{W^{1/2,2}(\R)}\\
 &\quad\aleq \|B(\eps_1) - B(\eps_2)\|_{L^2(\Ep_{od}^1\R)}[P]_{W^{1/2,2}(\R)} + [\eps_1 - \eps_2]_{W^{1/2,2}(\R)}\|\Omega^p\|_{L^2(\Ep_{od}^1 \R)} + \Lambda_{\eps_1,\eps_2}.
\end{split}
 \end{equation}

Combining \eqref{eq:Tdifferenceestimate} with \eqref{eq:Bdifferenceestimate} and \eqref{eq:Lambdaeps1eps2} we get
\[
 \begin{split}
 \|T(\eps_1) - T(\eps_2)\|_{L^\infty(\R)} + [&T(\eps_1) - T(\eps_2)]_{W^{1/2,2}(\R)}\\
 &\aleq \|\eps_1-\eps_2\|_{L^\infty(\R)} \brac{[P]_{W^{1/2,2}(\R\change{)}} + \|\Omega\|_{L^2(\Ep_{od}^1 \R)}}[P]_{W^{1/2,2}(\R)}\\
 &\quad + [\eps_1 - \eps_2]_{W^{1/2,2}(\R)}\brac{[P]_{W^{1/2,2}(\R\change{)}} + \|\Omega\|_{L^2(\Ep_{od}^1 \R)}}\\
 &\aleq (\|\eps_1 - \eps_2\|_{L^\infty(\R)} + [\eps_1 - \eps_2]_{W^{1/2,2}(\R)})\sigma,
 \end{split}
\]
where in the last inequality we used \eqref{eq:smallnessPandOmega}.

Thus, taking $\sigma$ small enough we obtain
\[
 \|T(\eps_1) - T(\eps_2)\|_{L^\infty(\R)} + [T(\eps_1) - T(\eps_2)]_{W^{1/2,2}(\R)} \le \lambda\brac{ \|\eps_1 - \eps_2\|_{L^\infty(\R)} + [\eps_1 - \eps_2]_{W^{1/2,2}(\R)}}, 
\]
for a $0<\lambda<1$, which implies that $T$ is a contraction. Consequently, by Banach fixed point theorem, there exists a unique $\eps \in X$, such that $T(\eps) = \eps$. That is we have a solution $T(\eps) = \eps$, which is a solution to
\[
\left\{ \begin{array}{l}
 -\brac{(I+\eps\change{(x)})\, \Omega^P P\change{(y)}}
\change{-} \brac{d_{\frac{1}{2}} \eps \,P\change{(y)}}
+R_\eps  = d_{\frac{1}{2}} a(\eps) + B(\eps)\\[3mm]
-\div_{\frac{1}{2}}\brac{(I+\eps\change{(x)})\, \Omega^P}
\change{-} \div_{\frac{1}{2}}\brac{d_{\frac{1}{2}} \eps}
+\div_{\frac{1}{2}}(R_\eps \,P^T\change{(y)}) = \div_{\frac{1}{2}} \brac{B(\eps)\,P^T\change{(y)}}.
       \end{array}
\right.
\]

Moreover, combining \eqref{eq:Tepsestimate} with \eqref{eq:Bestimate} and \eqref{eq:smallnessPandOmega} we obtain the following estimate on $\eps$
\begin{equation}\label{eq:eps-estimate1}
 \|\eps\|_{L^\infty(\R)} + [\eps]_{W^{\frac12,2}(\R)} \aleq \sigma\|\eps\|_{L^\infty(\R)} + \sigma[\eps]_{W^{\frac12,2}(\R)} +\|\Omega\|_{L^2(\Ep_{od}^1 \R)}+[P]_{W^{\frac12,2}(\R)},
\end{equation}
which gives for sufficiently small $\sigma$
\[
 \|\eps\|_{L^\infty(\R)} + [\eps]_{W^{\frac12,2}(\R)} \aleq \|\Omega\|_{L^2(\Ep_{od}^1 \R)}+[P]_{W^{\frac12,2}(\R)}.
\]
\end{proof}

\begin{proof}[Proof of \Cref{th:main}]
By \Cref{pr:nlocfixedpoint} we obtain the existence of an $\eps\in L^\infty\cap \dot{W}^{\frac12,2}(\R)$, $a\in \dot{W}^{\frac12,2}(\R)$, $B\in L^2(\Ep_{od}^1 \R)$ with $\div_\frac12 B =0$ satisfying the equations
solution $T(\eps) = \eps$, which is a solution to
\[
\left\{ \begin{array}{l}
 -\brac{(I+\eps\change{(x)})\, \Omega^P P\change{(y)}}
\change{-} \brac{d_{\frac{1}{2}} \eps \, P\change{(y)}}
+R_\eps  = d_{\frac{1}{2}} a + B\\
-\div_{\frac{1}{2}}\brac{(I+\eps\change{(x)})\, \Omega^P}
\change{-} \div_{\frac{1}{2}}\brac{d_{\frac{1}{2}} \eps}
+\div_{\frac{1}{2}}(R_\eps \,P^T\change{(y)}) = \div_{\frac{1}{2}} \brac{BP^T\change{(y)}},
       \end{array}
\right.
\]
where $P\in \dot{W}^{\frac12,2}(\R,SO(N))$ and $\Omega^P \in L^2(\Ep_{od}^1 \R)$ are taken from \Cref{th:gauge} and $[P]_{W^{1/2,2}(\R)} \aleq \|\Omega\|_{L^2(\Ep_{od}^1 \R)} \le \sigma$. 

By \Cref{le:aisvanishing} we have for sufficiently small $\sigma$
\[
 -\brac{(I+\eps\change{(x)})\, \Omega^P P\change{(y)}}
\change{-} \brac{d_{\frac{1}{2}} \eps \,P\change{(y)}}
+R_\eps  = B.
\]
Thus, defining for $\eps$ from \Cref{pr:nlocfixedpoint}, $A\coloneqq (I+\eps) P$, we have by  \Cref{la:rewrite1}
\[
 A\Omega - d_\frac12 A = B.
\]
The invertibility of $A$ follows from the invertibility of $P$ and $I+\eps$. Finally, since $A = (I+\eps)P$, we obtain from \eqref{eq:eps-estimates} and \eqref{eq:Pestimateinte} the estimates
\[
[A]_{W^{\frac12,2}(\R)} \aleq (1+\|\eps\|_{L^\infty}) [P]_{W^{\frac12,2}(\R)} + [\eps]_{W^{\frac12,2}(\R)} \aleq \|\Omega\|_{L^2(\Ep_{od}^1 \R)},
\]
and 
\[
 \|A\|_{L^\infty(\R)} \aleq 1 + \|\Omega\|_{L^2(\Ep_{od}^1 \R)}.
\]
This finishes the proof.

\end{proof}

\section{Weak convergence result --- Proof of Theorem~\ref{th:weakconvergence}}\label{s:weakconv}

\change{Using \Cref{la:prodrule} we obtain the following.}

\begin{lemma}\label{le:reformulateOmegatoA}
Assume that $\Omega\in L^2(\Ep_{od}^1 \R)$. Then $u\in \dot{W}^{\frac12,2}(\R,\R^N)\cap(L^2+L^\infty(\R))$ is a solution to 
\begin{equation}\label{eq:uomegaequation}
 \laph u^i = \Omega \cdot d_{\frac{1}{2}} u
\end{equation}
if and only if for any invertible matrix valued function $A,A^{-1}\in L^\infty\cap\dot{W}^{\frac12,2}(\R,GL(N))$,
\[
\div_{\frac{1}{2}} (A_{ik} d_{\frac{1}{2}} u^k) =\brac{A_{ij}\Omega_{j k}\, -d_{\frac{1}{2}} A_{ik} } \cdot d_{\frac{1}{2}}u^k.
  \]
\end{lemma}

In a first step we prove the ``local version'' of \Cref{th:weakconvergence}.
\begin{proposition}\label{pr:weakconv:1}
Let $\sigma > 0$ be the number from \Cref{th:main}.
Let $\{u_\ell\}_{\ell\in\N}$ be   a sequence  as in \Cref{th:weakconvergence} of solutions to
\[
 (-\Delta)^{\frac12} u_\ell = \Omega_{\ell} \cdot d_\frac12 u_\ell + f_\ell \quad \text{in } {\mathcal{D}}'(\R).
\]
Additionally let us assume that for some bounded interval $D \subset \R$ we have 
\begin{equation}\label{eq:OmegaboundD}
 \sup_{\ell} \|\Omega_\ell\|_{L^2(\Ep^1_{od} D)} < \sigma.
\end{equation}
Then
\[
 (-\Delta)^{\frac12} u = \Omega \cdot d_\frac12 u +f\quad \text{in } {\mathcal{D}}'(D).
\]
\end{proposition}
\begin{proof}
Let us define $\Omega_{D,\ell} \coloneqq \chi_{D}(x) \chi_{D}(y) \Omega_\ell \in L^2(\Ep_{od}^1 \R)$. Then by \eqref{eq:OmegaboundD} we have
\begin{equation}\label{eq:uniformboundonOmegaellD}
 \|\Omega_{D,\ell}\|_{L^2(\Ep_{od}^1 \R)} \le \|\Omega_\ell\|_{L^2(\Ep_{od}^1 D)} < \sigma.
\end{equation}
By \Cref{th:main} for $\Omega_{D,\ell}$ there exists a gauge $A_\ell$ such that 
\begin{equation}\label{eq:divfreeonD}
\div_\frac12(\Omega_{D,\ell}^{A_\ell}) =0,
\end{equation}
where $\Omega_{D,\ell}^{A_\ell}\coloneqq A_\ell \Omega_{D,\ell} - d_\frac12 A_\ell$.

Let $D_1  \subset \subset D$ be an open set.

By assumption and \Cref{le:reformulateOmegatoA} we have for any $\psi \in C_c^\infty(D_1)$ \change{and for $\Omega_\ell^{A_\ell}=A_\ell\Omega_\ell - d_{\frac12}A_\ell$}
\[
 \int_{\R} A_\ell d_\frac12 u_\ell \cdot d_\frac12 \psi= \int_{\R} \Omega_{\ell}^{ A_\ell} \change{\cdot \,} d_{\frac{1}{2}} u_\ell\, \psi+ f_\ell[ A_\ell \psi].
\]
\change{Here with a slight abuse of notation we write for the matrix product $\brac{f[A\psi]}^i \coloneqq \sum_{k} f^k[A^{ik} \psi ]$.}

Let us denote $\Omega_{D^c,\ell} \coloneqq \Omega_\ell - \Omega_{D,\ell}$. Then we have 
\[
 \int_{\R} A_\ell d_\frac12 u_\ell \cdot d_\frac12 \psi
 = \int_{\R} \Omega_{D,\ell}^{A_\ell}\cdot d_{\frac{1}{2}} u_\ell \, \psi
 +\int_{\R} A_\ell \Omega_{D^c,\ell}\cdot d_{\frac{1}{2}} u_\ell \, \psi+ f_\ell[A_\ell \psi].
\]
By \Cref{la:prodrule} \change{and \eqref{eq:divfreeonD} we have $\div_\frac12 \brac{\brac{\Omega_{D,\ell}^{A_\ell}}^\ast}=0$}, \change{thus again by \Cref{la:prodrule} we get $\Omega_{D,\ell}^{A_\ell}\cdot d_{\frac{1}{2}} u_\ell = \div_{\frac{1}{2}}\brac{\brac{\Omega_{D,\ell}^{A_\ell}}^\ast u_\ell(x)}$. Therefore,}
\begin{equation}\label{eq:equationbeforepassingtothelimit}
 \int_{\R} A_\ell d_\frac12 u_\ell \cdot d_\frac12 \psi= \int_{\R} \brac{\Omega_{D,\ell}^{A_\ell}}^\ast\cdot u_\ell d_{\frac{1}{2}}\psi+\int_{\R} A_\ell \Omega_{D^c,\ell}\cdot d_{\frac{1}{2}} u_\ell \, \psi+ f_\ell[A_\ell \psi].
\end{equation}
We will pass with $\ell\to\infty$ in \eqref{eq:equationbeforepassingtothelimit}. Roughly speaking, the convergence of most of the terms will be a result of a combination of weak-strong convergence. We first observe that by \Cref{th:main} we have
\[
 \|A_\ell\|_{\dot{W}^{\frac12,2}(\R)} \aleq \|\Omega_{D,\ell}\|_{L^2(\Ep_{od}^1 \R)} \le \sigma \quad \text{ and } \quad \|A_\ell\|_{L^\infty(\R)} \aleq 1+\sigma.
\]
Thus, $\sup_\ell \|A_\ell\|_{\dot{W}^{\frac12,2}(\R)} <\infty$ and $\sup_\ell \|A_\ell\|_{L^\infty}(\R) < \infty$. Up to taking a subsequence  we obtain
\begin{equation}\label{eq:Aconvergence}
  A_\ell \rightharpoonup A \quad \text{ weakly in }\dot{W}^{\frac12,2}(\R,\R^N), \quad A_\ell \rightarrow A \quad \text{ locally strongly in }L^2,
\end{equation}
where we used the Rellich--Kondrachov's compact embedding theorem and $A\in L^\infty \cap \dot{W}^{\frac12,2}(\R,GL(N))$. By the pointwise a.e. convergence we have $\|A\|_{L^\infty(\R)}\aleq 1+ \sigma$.

By \eqref{eq:uniformboundonOmegaellD} we also have up to a subsequence
\[
 \Omega_{D,\ell} \rightharpoonup \Omega_D \quad \text{ weakly in }L^2(\Ep_{od}^1 \R),
\]
where $\Omega_D\in L^2(\Ep_{od}^1 \R)$. 

By assumptions of the Theorem we also have, up to a subsequence,
\[
u_\ell \rightharpoonup u \quad \text{ weakly in } \dot{W}^{\frac12,2}(\R), \quad u_\ell \to u \quad \text{ locally strongly in }L^2, 
\]
where $u\in  \dot{W}^{\frac12,2}(\R,\R^N)$.

\change{Let us choose a large $R \gg 1$, such that in particular $D_1 \subset B(R)$. We begin with the first term of \eqref{eq:equationbeforepassingtothelimit}.}

\underline{\hypertarget{Step1}{\textsc{Step 1.}}} We claim that (up to a subsequence)
\begin{equation}\label{eq:limitI}
 \lim_{\ell \to \infty} \int_{\R} A_\ell d_\frac12 u_\ell \cdot d_\frac12 \psi = \int_{\R} A d_\frac12 u \cdot d_\frac12 \psi.
\end{equation}
Indeed, we observe 
\begin{equation}\label{eq:weakstrongconvergence}
\begin{split}
 \int_{\R} A_\ell d_\frac12 u_\ell \cdot d_\frac12 \psi - \int_{\R} A d_\frac12 u \cdot d_\frac12 \psi 
 &= \int_{\R} (A_\ell - A)d_\frac12 u_\ell \cdot d_\frac12 \psi + \int_\R A (d_\frac12 u_\ell - d_\frac12 u) \cdot d_\frac12 \psi.
\end{split}
 \end{equation}
By weak convergence of $d_\frac12 u_\ell$ in $L^2(\Ep_{od}^1 \R)$ we have
\begin{equation}\label{eq:eqeqeqeqeqe}
 \lim_{\ell\to\infty} \int_\R A (d_\frac12 u_\ell - d_\frac12 u) \cdot d_\frac12 \psi =0.
\end{equation}
As for the first term on the right-hand side of \eqref{eq:weakstrongconvergence} we observe that since $\supp \psi \subset D_1\subset B(R)$,
\begin{equation}\label{eq:splitwithBR}
\begin{split}
 &\int_{\R}\int_{\R} (A_\ell(x) - A(x)) \frac{(u_\ell(x)-u_\ell(y)) (\psi(x)-\psi(y))}{|x-y|^{2}}\dx \dy\\
 &=
 \int_{B(R)}\int_{B(R)} (A_\ell(x) - A(x)) \frac{(u_\ell(x)-u_\ell(y)) (\psi(x)-\psi(y))}{|x-y|^{2}}\dx \dy\\
 &\quad+
 \int_{\R}\int_{B(R)} (A_\ell(x) - A(x)) \frac{(u_\ell(x)-u_\ell(y)) (\psi(x)-\psi(y))}{|x-y|^{2}}\dx \dy\\
 &\quad+\int_{B(R)}\int_{\R \setminus B(R)} (A_\ell(x) - A(x)) \frac{(u_\ell(x)-u_\ell(y)) (\psi(x)-\psi(y))}{|x-y|^{2}}\dx \dy.
 \end{split}
\end{equation}
By strong convergence in $L^2$ of $A_\ell$ on compact domains, we have 
 \begin{equation}\label{eq:firstguyyay}
 \begin{split}
  &\lim_{\ell \to \infty} \int_{B(R)}\int_{B(R)} (A_\ell(x) - A(x)) \frac{(u_\ell(x)-u_\ell(y)) (\psi(x)-\psi(y))}{|x-y|^{2}}\dx \dy\\
  &\quad \aleq \lim_{\ell \to \infty} \|A_\ell - A\|_{L^2(B(R))}\|\psi\|_{\lip}[u_\ell]_{W^{\frac12,2}(B(R))} =0
  \end{split}
 \end{equation}
 and (noting once again that $\supp \psi \subset D_1$)
\begin{equation}\label{eq:secondguytoo}
\begin{split}
 &\lim_{\ell \to \infty} \abs{\int_{\R\setminus B(R)}\int_{B(R)} (A_\ell(x) - A(x)) \frac{(u_\ell(x)-u_\ell(y))(\psi(x)-\psi(y))}{|x-y|^{2}}\dx \dy}\\
&\quad \aleq \lim_{\ell \to \infty}\change{\|A_\ell-A\|_{L^2(B(R))} [u_\ell]_{W^{\frac12,2}(\R)} \brac{\int_{\R\setminus B(R)} \sup_{x\in D_1} \frac{|\psi(x) - \psi(y)|^2}{|x-y|^2} \dif y}^\frac12}\\
&\quad \change{\aleq \lim_{\ell \to \infty}\|A_\ell-A\|_{L^2(B(R))} [u_\ell]_{W^{\frac12,2}(\R)} \|\psi\|_{L^\infty} \brac{\int_{\R\setminus B(R)} \frac{1}{1+|y|^2} \dif y}^\frac12} =0.
 \end{split}
\end{equation}
\change{In the last inequality we used the fact that if $x\in D_1$ and $y\in \R\setminus B(R)$ then $|x-y|\ageq 1 + |y|$.}

For the last term of \eqref{eq:splitwithBR}, we \change{similarly use} that if $y \in \supp \psi$ and $x \in \R \setminus B(R)$, then we have $|x-y| \ageq 1+|x|$ with a constant independent of $R$.
\[
\begin{split}
 &\abs{\int_{B(R)}\int_{\R \setminus B(R)} (A_\ell(x) - A(x)) \frac{(u_\ell(x)-u_\ell(y))
(\psi(x)-\psi(y))}{|x-y|^{2}}\dx \dy}\\
&\aleq \brac{\|A_\ell\|_{L^\infty} + \|A\|_{L^\infty}} \|\psi\|_{L^\infty} \int_{D_1}\int_{\R \setminus B(R)} \frac{|u_\ell(x)|+|u_\ell(y)|}{1+|x|^{2}}\dx \dy\\
&\aleq \brac{\|A_\ell\|_{L^\infty} + \|A\|_{L^\infty}} \|\psi\|_{L^\infty}\, \brac{\|u_\ell\|_{L^2(D_1)} \int_{\R \setminus B(R)} \frac{1}{1+|x|^2} \dx}\\
&\quad+ \brac{\|A_\ell\|_{L^\infty} + \|A\|_{L^\infty}} \|\psi\|_{L^\infty}\,\brac{\|u_\ell\|_{L^\infty+L^2(\R)} \brac{\int_{\R \setminus B(R)} \frac{\dx}{1+|x|^2} + \brac{\int_{\R \setminus B(R)} \frac{\dx}{(1+|x|^2)^2}}^{\frac{1}{2}}}}\\
&\aleq \brac{\|A_\ell\|_{L^\infty} + \|A\|_{L^\infty}} \|\psi\|_{L^\infty}\, \brac{\|u\|_{L^2(D_1)} + \|u_\ell\|_{L^\infty+L^2(\R)} } R^{-\frac12}.
\end{split}
\]
So we have 
\begin{equation}\label{eq:thirdguytoo}
 \lim_{R \to \infty} \sup_{\ell} \abs{\int_{B(R)}\int_{\R \setminus B(R)} (A_\ell(x) - A(x)) \frac{(u_\ell(x)-u_\ell(y))
(\psi(x)-\psi(y))}{|x-y|^{2}}\dx \dy} = 0 \change{.}
\end{equation} 
\change{By} \eqref{eq:splitwithBR}, \eqref{eq:firstguyyay},
\eqref{eq:secondguytoo}, and \eqref{eq:thirdguytoo} we obtain the convergence of the \change{first term on the right-hand side} of \eqref{eq:weakstrongconvergence}, i.e.,
\begin{equation}\label{eq:theconvergenceofthefirstguy}
 \lim_{\ell \to \infty} \int_{\R}\int_{\R} (A_\ell(x) - A(x)) \frac{(u_\ell(x)-u_\ell(y)) (\psi(x)-\psi(y))}{|x-y|^{2}}\dx \dy = 0.
\end{equation}
Thus, combining \eqref{eq:weakstrongconvergence}, \eqref{eq:eqeqeqeqeqe}, and \eqref{eq:theconvergenceofthefirstguy} we obtain the claim \eqref{eq:limitI}.

\underline{\textsc{Step 2.}} We claim that (up to a subsequence)
\begin{equation}\label{eq:secondtermdidIusethisname}
 \lim_{\ell\to\infty}\int_{\R} \brac{\Omega_{D,\ell}^{A_\ell}}^\ast\cdot u_\ell d_{\frac{1}{2}}\psi = \int_{\R} \brac{\Omega_{D}^{A}}^\ast\cdot u d_{\frac{1}{2}}\psi,
\end{equation}
where 
$\Omega_D^A \coloneqq A\Omega_D - d_\frac12 A$.

Indeed, we write
\begin{equation}\label{eq:steptwofirstsplit}
 \begin{split}
  &\int_{\R} \brac{\Omega_{D,\ell}^{A_\ell}}^\ast\cdot u_\ell d_{\frac{1}{2}}\psi -\int_{\R} \brac{\Omega_{D}^{A}}^\ast\cdot u d_{\frac{1}{2}}\psi\\
  &= \int_\R \int_\R \brac{\brac{\Omega_{D,\ell}^{A_\ell}}^\ast(x,y) u_\ell(x) - \brac{\Omega_{D}^{A}}^\ast(x,y)u(x)}\frac{\psi(x)-\psi(y)}{|x-y|^\frac12}\frac{\dif x \dif y}{|x-y|}\\
  &= \change{\int_\R \int_\R \brac{A_\ell(y)\Omega_{D,\ell}(y,x)u_\ell(x) - A(y)\Omega_D(y,x)u(x)} \frac{\psi(x)-\psi(y)}{|x-y|^\frac12}\frac{\dif x \dif y}{|x-y|}}\\
  &\quad\change{- \int_\R \int_\R \brac{d_\frac12 A_\ell(y,x)u_\ell(x) - d_\frac12 A(y,x)u(x)}\frac{\psi(x)-\psi(y)}{|x-y|^\frac12}\frac{\dif x \dif y}{|x-y|}}.
 \end{split}
\end{equation}
\change{Now, in order to obtain
\begin{equation}\label{eq:Step21}
 \lim_{\ell\to0}\int_\R \int_\R \brac{A_\ell(y)\Omega_{D,\ell}(y,x)u_\ell(x) - A(y)\Omega_D(y,x)u(x)} \frac{\psi(x)-\psi(y)}{|x-y|^\frac12}\frac{\dif x \dif y}{|x-y|}=0
\end{equation}
we split the integral in two}
\begin{equation}\label{eq:Step211}
 \begin{split}
&\change{\int_\R \int_\R \brac{A_\ell(y)\Omega_{D,\ell}(y,x)u_\ell(x) - A(y)\Omega_D(y,x)u(x)} \frac{\psi(x)-\psi(y)}{|x-y|^\frac12}\frac{\dif x \dif y}{|x-y|}}\\
&\quad \change{= \int_\R \int_\R \brac{A(y)u(x)(\Omega_{D,\ell}(y,x)-\Omega_D(y,x))} \frac{\psi(x)-\psi(y)}{|x-y|^\frac12}\frac{\dif x \dif y}{|x-y|}}\\
&\quad \quad \change{+ \int_\R \int_\R \brac{A_\ell(y)u_\ell(x) - A(y)u(x)}\Omega_{D,\ell}(y,x) \frac{\psi(x)-\psi(y)}{|x-y|^\frac12}\frac{\dif x \dif y}{|x-y|}.}
 \end{split}
\end{equation}
\change{The first term on the right-hand side of \eqref{eq:Step211} converges to zero as $\ell\to\infty$. This follows from the weak convergence of $\Omega_{D,\ell} \rightharpoonup \Omega_D$ in $L^2(\Ep^1_{od}\R)$, the fact that $\Omega_{D,\ell} - \Omega_D$ is supported on $D\times D$, and that $A(y)u(x)d_\frac12\psi(x,y)\chi_{D}(x) \chi_{D}(y) \in L^2(\Ep^1_{od}\R)$ (the easy verification of the latter is left to the reader).}

\change{As for the second term on the right-hand side of \eqref{eq:Step211} we begin with the observation that} 
\begin{equation}\label{eq:splisplitsplit}
\begin{split}
&\change{\int_\R \int_\R \brac{A_\ell(y)u_\ell(x) - A(y)u(x)}\Omega_{D,\ell}(y,x) \frac{\psi(x)-\psi(y)}{|x-y|^\frac12}\frac{\dif x \dif y}{|x-y|}}\\
&\change{= \int_\R \int_\R \brac{A_\ell(y) - A(y)}u_\ell(x)\Omega_{D,\ell}(y,x) \frac{\psi(x)-\psi(y)}{|x-y|^\frac12}\frac{\dif x \dif y}{|x-y|}}\\
&\quad\change{+ \int_\R \int_\R A(y)(u_\ell(x) - u(x))\Omega_{D,\ell}(y,x) \frac{\psi(x)-\psi(y)}{|x-y|^\frac12}\frac{\dif x \dif y}{|x-y|}.}
\end{split}
\end{equation}
\change{To estimate the first term of the right-hand side of \eqref{eq:splisplitsplit} we first note that the support of $\Omega_{D,\ell}$ is $D\times D$ and then we use H\"{o}lder's inequality}
\begin{equation}\label{eq:oneoneone}
\begin{split}
 &\change{\lim_{\ell\to \infty }\abs{\int_{\R} \int_{\R}(A_\ell(y) - A(y))u_\ell(x)\Omega_{D,\ell}(y,x) \frac{\psi(x)-\psi(y)}{|x-y|^\frac12}\frac{\dif x \dif y}{|x-y|}}}\\
 &\change{\le \lim_{\ell\to \infty}\int_{D} \int_{D}\abs{A_\ell(y) - A(y)}\abs{u_\ell(x)}\abs{\Omega_{D,\ell}(y,x)} \frac{\abs{\psi(x)-\psi(y)}}{|x-y|^\frac12}\frac{\dif x \dif y}{|x-y|}}\\
 & \change{\le \lim_{\ell\to \infty} \|A_\ell - A\|_{L^2(D)} \|u_\ell\|_{L^2(D)}\|\Omega_{D,\ell}\|_{L^2(\Ep^1_{od}\R)} \|\psi\|_{\lip}=0.}
\end{split}
 \end{equation}
\change{Now we verify the convergence of the second term of the right-hand side of \eqref{eq:splisplitsplit}. Again we use that the support of $\Omega_{D,\ell}$ is $D\times D$ and thus by the strong convergence in $L^2$ of $u_\ell$ on compact domains we have}
\begin{equation}\label{eq:hops}
 \begin{split}
 &\change{\lim_{\ell\to\infty}\abs{\int_\R \int_\R A(y)(u_\ell(x) - u(x))\Omega_{D,\ell}(y,x) \frac{\psi(x)-\psi(y)}{|x-y|^\frac12}\frac{\dif x \dif y}{|x-y|}}}\\
  &\change{\le \lim_{\ell\to\infty}\int_{D}\int_{D}\abs{A(y)(u_\ell(x) - u(x))\Omega_{D,\ell}(y,x)} \frac{\abs{\psi(x)-\psi(y)}}{|x-y|^\frac12}\frac{\dif x \dif y}{|x-y|}}\\
  &\change{\aleq \lim_{\ell\to\infty} \|A\|_{L^\infty}\|u_\ell - u\|_{L^2(D)} \|\Omega_{D,\ell}\|_{L^2(\Ep_{od}^1 \R)}\|\psi\|_{\lip}=0.} 
 \end{split}
\end{equation}

\change{We also claim that
\begin{equation}\label{eq:Step22}
\lim_{\ell\to\infty} \int_\R \int_\R \brac{d_\frac12 A_\ell(y,x)u_\ell(x) - d_\frac12 A(y,x)u(x)}\frac{\psi(x)-\psi(y)}{|x-y|^\frac12}\frac{\dif x \dif y}{|x-y|} =0.
\end{equation}}

\change{To verify this statement we divide the integral in two}
\begin{equation}\label{eq:step2secondterm}
\begin{split}
 &\change{\int_\R \int_\R \brac{d_\frac12 A_\ell(y,x)u_\ell(x) - d_\frac12 A(y,x)u(x)}\frac{\psi(x)-\psi(y)}{|x-y|^\frac12}\frac{\dif x \dif y}{|x-y|}}\\
 &\change{= \int_\R \int_\R d_\frac12 A_\ell(y,x)(u_\ell(x) - u(x))\frac{\psi(x)-\psi(y)}{|x-y|^\frac12}\frac{\dif x \dif y}{|x-y|}} \\
 &\quad \change{+ \int_\R \int_\R (d_\frac12 A_\ell(y,x) - d_\frac12A(y,x))u(x)\frac{\psi(x)-\psi(y)}{|x-y|^\frac12}\frac{\dif x \dif y}{|x-y|}.}
\end{split}
\end{equation}
\change{The second term on the right-hand side of \eqref{eq:step2secondterm} converges to zero as $\ell\to\infty$, because $d_\frac12 A\rightharpoonup d_\frac12 A$ weakly in $L^2(\Ep^1_{od} \R)$ and $u(x)d_\frac12 \psi(x,y) \in L^2(\Ep^1_{od}\R)$.}

\change{We verify the convergence of the first term on the right-hand side of \eqref{eq:step2secondterm}. First we note that by the strong convergence of $u_\ell$ in $L^2$ on compact domains we have}
\begin{equation}\label{eq:111}
\begin{split}
 &\change{\lim_{\ell\to \infty}\int_{B(R)}\int_{B(R)} d_\frac12 A_\ell(y,x)(u_\ell(x) - u(x))\frac{\psi(x)-\psi(y)}{|x-y|^\frac12}\frac{\dif x \dif y}{|x-y|}}\\
 &\change{\le \|u_\ell - u\|_{L^2(B(R))} \|\psi\|_{\lip}[A_\ell]_{W^{\frac12,2}(\R)} =0}   
\end{split}
\end{equation}
\change{and}
\begin{equation}\label{eq:222}
\begin{split}
 &\change{\lim_{\ell\to \infty}\int_{\R\setminus B(R)} \int_{B(R)}d_\frac12 A_\ell(y,x)(u_\ell(x) - u(x))\frac{\psi(x)-\psi(y)}{|x-y|^\frac12}\frac{\dif x \dif y}{|x-y|}}\\
 &\change{\aleq \lim_{\ell\to \infty}\|u_\ell - u\|_{L^2(B(R))}[A_\ell]_{W^{\frac12,2}(\R)}\|\psi\|_{L^\infty}\brac{\int_{\R\setminus B(R)} \frac{1}{1+|y|^2} \dif y}^\frac12 =0.} 
\end{split}
 \end{equation}
 \change{Finally, we have since $\supp \psi \subset D_1 \subset B(R)$}
 \begin{equation}
 \begin{split}
  &\change{\int_{B(R)} \int_{\R\setminus B(R)}d_\frac12 A_\ell(y,x)(u_\ell(x) - u(x))\frac{\psi(x)-\psi(y)}{|x-y|^\frac12}\frac{\dif x \dif y}{|x-y|}}\\
  &\change{\aleq [A_\ell]_{W^{\frac12,2}(\R)} \|\psi\|_{L^\infty}\brac{\int_{\R\setminus B(R)}\frac{|u_\ell(x) - u(x)|^2}{1+|x|^2}\dif x}^\frac12}\\
 &\change{\aleq [A_\ell]_{W^{\frac12,2}(\R)}\|\psi\|_{L^\infty}\|u_\ell - u\|_{L^2+L^\infty(\R)}\max\left\{\brac{\int_{\R\setminus B(R)} \frac{1}{1+|x|^2}\dif x}^\frac12, \brac{\frac{1}{1+R^2}}^\frac12\right\}}\\
  &\change{\aleq R^{-\frac12}[A_\ell]_{W^{\frac12,2}(\R)} \|\psi\|_{L^\infty}\|u_\ell - u\|_{L^2+L^\infty(\R)}.}
 \end{split}
 \end{equation}
\change{This gives
\begin{equation}\label{eq:333}
 \lim_{R\to \infty} \sup_\ell \abs{\int_{B(R)} \int_{\R\setminus B(R)}d_\frac12 A_\ell(y,x)(u_\ell(x) - u(x))\frac{\psi(x)-\psi(y)}{|x-y|^\frac12}\frac{\dif x \dif y}{|x-y|}} =0.
\end{equation}
Thus the convergence of the first term of \eqref{eq:step2secondterm} follows from \eqref{eq:111}, \eqref{eq:222}, and \eqref{eq:333}. We proved \eqref{eq:Step22}.}

\change{
Now \eqref{eq:secondtermdidIusethisname} follows from \eqref{eq:steptwofirstsplit} combined with \eqref{eq:Step21} and \eqref{eq:Step22}.}

\underline{\textsc{Step 3.}} We claim that
\begin{equation}\label{eq:divfreeOmegaDA}
\div_\frac12 \brac{\Omega_{D}^A}^\ast =0. 
\end{equation}
That is, we claim that for any $\vp\in C_c^\infty(\R)$ we have
\[
0=\lim_{\ell\to \infty} \int_{\R} \int_{\R} \brac{\Omega_{D,\ell}^{A_\ell}}^\ast \frac{\vp(x) - \vp(y)}{|x-y|^\frac12} \frac{\dif x \dif y}{|x-y|} =  \int_{\R} \int_{\R}  \brac{\Omega_{D}^{A}}^\ast \frac{\vp(x) - \vp(y)}{|x-y|^\frac12} \frac{\dif x \dif y}{|x-y|}.
\]
We write
\begin{equation}\label{eq:ecolabel}
 \begin{split}
  &\int_{\R} \int_{\R} \brac{\Omega_{D,\ell}^{A_\ell}}^\ast \frac{\vp(x) - \vp(y)}{|x-y|^\frac12} \frac{\dif x \dif y}{|x-y|} -  \int_{\R} \int_{\R}  \brac{\Omega_{D}^{A}}^\ast \frac{\vp(x) - \vp(y)}{|x-y|^\frac12} \frac{\dif x \dif y}{|x-y|}\\
  &= \int_\R \int_\R \brac{A(y)\Omega_D(y,x) - A_\ell(y)\Omega_{D,\ell}(y,x)}d_\frac12 \vp(x,y) \frac{\dif x \dif y}{|x-y|}\\
  &\quad+ \int_\R \int_\R \brac{d_\frac12 A(y,x) - d_\frac12 A_\ell(y,x)}d_\frac12 \vp(x,y) \frac{\dif x \dif y}{|x-y|}.
 \end{split}
\end{equation}
As for the second term of \eqref{eq:ecolabel} we observe that by weak convergence of $d_\frac12 A_\ell$ in $L^2(\Ep_{od}^1 \R)$ we have
\[
 \lim_{\ell\to\infty} \int_\R \int_\R \brac{d_\frac12 A(y,x) - d_\frac12 A_\ell(y,x)}d_\frac12 \vp(x,y) \frac{\dif x \dif y}{|x-y|} =0.
\]
As for the first term of \eqref{eq:ecolabel} we proceed exactly as in \hyperlink{Step1}{Step 1} and obtain
\[
 \lim_{\ell\to\infty} \int_\R \int_\R \brac{A(y)\Omega_D(y,x) - A_\ell(y)\Omega_{D,\ell}(y,x)}d_\frac12 \vp(x,y) \frac{\dif x \dif y}{|x-y|} =0.
\]
This finishes the proof of \eqref{eq:divfreeOmegaDA}.

\underline{\textsc{Step 4.}} We claim that (up to a subsequence)
\begin{equation}\label{eq:number3}
 \lim_{\ell\to\infty} A_\ell \Omega_{D^c,\ell} \cdot d_\frac12 u_\ell \psi = \int_\R A\Omega_{D^c} \cdot d_\frac12 u\psi, 
\end{equation}
where $\Omega_{D^c} = \Omega - \Omega_D$ and $\Omega\in L^2(\Ep_{od}^1 \R)$ is the one given in the assumptions of the theorem. 
 
Indeed, since $\Omega_{D^c,\ell}(x,y)=0$ whenever both $x,y \in D$ we have by the support of $\psi$,
\begin{equation}\label{eq:decompositionOmegaDcellandFell}
\begin{split}
 \int_{\R} A_\ell\Omega_{D^c,\ell}\cdot d_{\frac{1}{2}} u \psi 
 &= \int_{\R} \int_{\R}  (A_\ell(x))_{ij} (\Omega_{D^c,\ell})_{jk}(x,y) \frac{\brac{u_\ell^k (x)-u_\ell^k(y)}}{|x-y|^\frac{1}{2}}\psi(x)\chi_{|x-y| \geq \dist(D_1,\partial D)} \frac{\dx \dy}{|x-y|}\\
 &= \int_{\R} \int_{\R} (\Omega_{D^c,\ell})_{jk}(x,y) (A_\ell(x))_{ij} \frac{\brac{u_\ell^k (x)-u_\ell^k(y)}}{|x-y|^\frac{1}{2}}\psi(x)\chi_{|x-y| \geq \dist(D_1,\partial D)} \frac{\dx \dy}{|x-y|}.
\end{split}
 \end{equation}
\change{We set}
\[
 F_\ell(x,y) \coloneqq  \chi_{|x-y| \geq \dist(D_1,\partial D)}\frac{\brac{u_\ell (x)-u_\ell(y)} }{|x-y|^\frac{1}{2}}  A_\ell (x) \psi(x) 
\]
and 
\[
 F(x,y)\coloneqq \chi_{|x-y| \geq \dist(D_1,\partial D)}\frac{\brac{u(x)-u(y)}}{|x-y|^\frac{1}{2}}  A(x) \psi(x).
\]
We claim that we have the strong convergence
\begin{equation}\label{eq:strongconvF}
\lim_{\ell\to\infty} \|F_\ell - F\|_{L^{2}(\Ep^1_{od} \R)} = 0.
\end{equation}

Indeed, we have
\begin{equation}\label{eq:splitdubrovnik}
 \begin{split}
   &\int_{\R}\int_{\R}|F_\ell(x,y)-F(x,y)|^2 \frac{\dx \dy}{|x-y|}\\
  &\le \int_{\R}\int_{D_1}\abs{d_\frac12 u_\ell(x,y) A_\ell(x) - d_\frac12 u(x,y)A(x)}^2{|\psi(x)|^2}\chi_{|x-y| \geq \dist(D_1,\partial D)} \frac{\dif x \dif y}{|x-y|}\\
&\aleq \int_{\R}\int_{D_1} \abs{d_\frac12 u_\ell(x,y) - d_\frac12 u(x,y)}^2 \brac{|A(x)|^2+|A_\ell(x)|^2} {|\psi(x)|^2}\chi_{|x-y| \geq \dist(D_1,\partial D)} \frac{\dif x \dif y}{|x-y|}\\
  &\quad + \int_{\R}\int_{D_1} \abs{d_\frac12 u(x,y)}^2|A_\ell(x) - A(x)|^2 |\psi(x)|^2 \chi_{|x-y| \geq \dist(D_1,\partial D)} \frac{\dif x \dif y}{|x-y|}.
 \end{split}
\end{equation}

\underline{For the first term of the right-hand side of \eqref{eq:splitdubrovnik}} we take $R \gg 1$, such that in particular $\supp \psi \subset D_1 \subset \subset D \subset B(R)$ and estimate
\begin{equation}\label{eq:splitdubrovnik2}
 \begin{split}
  &\int_{\R}\int_{D_1} \abs{d_\frac12 u_\ell(x,y) - d_\frac12 u(x,y)}^2 \brac{|A(x)|^2 + |A_\ell(x)|^2} {|\psi(x)|^2}\chi_{|x-y| \geq \dist(D_1,\partial D)} \frac{\dif x \dif y}{|x-y|}\\
  &=\int_{\R\setminus B(R)}\int_{D_1} \abs{d_\frac12 u_\ell(x,y) - d_\frac12 u(x,y)}^2 \brac{|A(x)|^2 + |A_\ell(x)|^2} {|\psi(x)|^2}\chi_{|x-y| \geq \dist(D_1,\partial D)} \frac{\dif x \dif y}{|x-y|}\\
  &\quad + \int_{B(R)}\int_{D_1} \abs{d_\frac12 u_\ell(x,y) - d_\frac12 u(x,y)}^2 \brac{|A(x)|^2 + |A_\ell(x)|^2} {|\psi(x)|^2}\chi_{|x-y| \geq \dist(D_1,\partial D)} \frac{\dif x \dif y}{|x-y|}.
 \end{split}
\end{equation}

Now, for the second term of the right-hand side of \eqref{eq:splitdubrovnik2} we have
\[
\begin{split}
&\int_{B(R)}\int_{D_1} \abs{d_\frac12 u_\ell(x,y) - d_\frac12 u(x,y)}^2 \brac{|A(x)|^2 + |A_\ell(x)|^2} {|\psi(x)|^2}\chi_{|x-y| \geq \dist(D_1,\partial D)} \frac{\dif x \dif y}{|x-y|}\\
 &\aleq \brac{\|A\|_{L^\infty(D_1)}^2+\|A_\ell\|_{L^\infty(D_1)}^2}\|\psi\|^2_{L^\infty}\int_{\change{B(R)}}\int_{D_1} \frac{|u_\ell(x) - u(x)|^2 + |u_\ell(y)-u(y)|^2}{|x-y|^2} \dif x \dif y\\
 &\aleq \brac{\|A\|_{L^\infty(D_1)}^2+\|A_\ell\|_{L^\infty(D_1)}^2}\|\psi\|^2_{L^\infty}\dist^{-2}(D_1,\partial D)\bigg(\int_{\change{B(R)}}\int_{D_1} |u_\ell(x) - u(x)|^2 \dif x \dif y\\
 &\quad\quad  + \int_{\change{B(R)}}\int_{D_1}|u_\ell(y)-u(y)|^2\dif x \dif y\bigg)\\
 &\le C(D_1,D,R) \brac{\|A\|_{L^\infty(D_1)}^2+\|A_\ell\|_{L^\infty(D_1)}^2}\|u_\ell - u\|^2_{L^2(B(R))}.
 \end{split}
\]
Thus, by the strong convergence on compact sets of $u_\ell$ in $L^2$ we obtain
\begin{equation}\label{eq:eqhowtonameit}
 \lim_{\ell\to\infty} \int_{B(R)}\int_{D_1} \abs{d_\frac12 u_\ell(x,y) - d_\frac12 u(x,y)}^2 \brac{|A(x)|^2 + |A_\ell(x)|^2} {|\psi(x)|^2}\chi_{|x-y| \geq \dist(D_1,\partial D)} \frac{\dif x \dif y}{|x-y|} =0.
\end{equation}

Now we estimate the first term of the right-hand side of \eqref{eq:splitdubrovnik2}. We observe that for all large $R$, whenever $x \in \supp \psi$ and $y \not \in B(R)$, we have $|x-y| \ageq 1+|y|$. Therefore, 
\begin{equation*}\label{eq:labelbaleblale}
\begin{split}
 &\int_{\R\setminus B(R)}\int_{D_1} \abs{d_\frac12 u_\ell(x,y) - d_\frac12 u(x,y)}^2 \brac{|A(x)|^2 + |A_\ell(x)|^2} {|\psi(x)|^2}\chi_{|x-y| \geq \dist(D_1,\partial D)} \frac{\dif x \dif y}{|x-y|}\\
 &\aleq \brac{\|A\|_{L^\infty(D_1)}^2+\|A_\ell\|_{L^\infty(D_1)}^2}\|\psi\|^2_{L^\infty} \int_{\R\setminus B(R)}\int_{D_1} \frac{|u_\ell(x) - u(x)|^2 + |u_\ell(y) - u(y)|^2}{1+|y|^2} \dx \dy\\
 &\aleq \brac{\|A\|_{L^\infty(D_1)}^2+\|A_\ell\|_{L^\infty(D_1)}^2}\|\psi\|^2_{L^\infty} \|u_\ell - u\|_{L^2(D_1)}^2 \int_{\R\setminus B(R)} \frac{1}{1+|y|^2} \dif y\\
 &\quad + \brac{\|A\|_{L^\infty(D_1)}^2+\|A_\ell\|_{L^\infty(D_1)}^2}\|\psi\|^2_{L^\infty} \|u_\ell - u\|_{L^2 + L^\infty(\R)}^{\change{2}}\max\left\{\int_{\R\setminus B(R)}\frac{1}{1+|y|^2}\dif y , \frac{1}{1+R^2}\right\}\\
 &\aleq R^{\change{-1}} \brac{\|A\|_{L^\infty(D_1)}^2+\|A_\ell\|_{L^\infty(D_1)}^2}\|\psi\|^2_{L^\infty}  \|u_\ell - u\|^{\change{ 2}}_{L^2 + L^\infty(\R)}.
 \end{split}
\end{equation*}
Thus,
\begin{equation}\label{eq:splitdubrovnik22}
 \lim_{R\to\infty} \sup_{\ell} \int_{\R\setminus B(R)}\int_{D_1} \abs{d_\frac12 u_\ell(x,y) - d_\frac12 u(x,y)}^2 \brac{|A(x)|^2 + |A_\ell(x)|^2} {|\psi(x)|^2}\chi_{|x-y| \geq \dist(D_1,\partial D)} \frac{\dif x \dif y}{|x-y|} =0.
\end{equation}

Combining \eqref{eq:splitdubrovnik2} with \eqref{eq:eqhowtonameit} and \eqref{eq:splitdubrovnik22} we obtain the convergence of the first term of the right-hand side of \eqref{eq:splitdubrovnik}
\begin{equation}\label{eq:toname}
 \lim_{\ell\to\infty}\int_{\R}\int_{D_1} \abs{d_\frac12 u_\ell(x,y) - d_\frac12 u(x,y)}^2 \brac{|A(x)|^2 + \change{|A_\ell(x)|^2}} {|\psi(x)|^2}\chi_{|x-y| \geq \dist(D_1,\partial D)} \frac{\dif x \dif y}{|x-y|} =0.
\end{equation}

\underline{As for the second term of the right-hand side of \eqref{eq:splitdubrovnik}} \change{we observe that since $A_\ell \to A$ pointwise almost everywhere, we have
\[
 \lim_{\ell\to\infty}\abs{d_\frac12 u(x,y)}^2|A_\ell(x) - A(x)|^2 |\psi(x)|^2 \frac{\chi_{|x-y| \geq \dist(D_1,\partial D)} }{|x-y|} = 0 \quad \text{ pointwise a.e. in } D_1 \times \R.
\]
Moreover, we have
\[
\begin{split}
 &\abs{d_\frac12 u(x,y)}^2|A_\ell(x) - A(x)|^2 |\psi(x)|^2 \chi_{|x-y| \geq \dist(D_1,\partial D)} \frac{1}{|x-y|}\\
 &\aleq \brac{\sup_{\ell} \|A_\ell\|^{\change{2}}_{L^\infty}+ \| A\|^{\change{2}}_{L^\infty}} \abs{d_\frac12 u(x,y)}^2 |\psi(x)|^2 \chi_{|x-y| \geq \dist(D_1,\partial D)} \frac{1}{|x-y|}
 \end{split}
\]
and the right-hand side is independent of $\ell$ and integrable. Thus, by dominated convergence theorem we have
\begin{equation}\label{eq:zagrebisthecapitol}
 \lim_{\ell \to \infty} \int_{\R}\int_{D_1} \abs{d_\frac12 u(x,y)}^2|A_\ell(x) - A(x)|^2 |\psi(x)|^2 \chi_{|x-y| \geq \dist(D_1,\partial D)} \frac{\dif x \dif y}{|x-y|} = 0.
\end{equation}
}

Now, plugging \eqref{eq:zagrebisthecapitol} and \eqref{eq:toname} into \eqref{eq:splitdubrovnik} we establish \eqref{eq:strongconvF}.

Thus, \eqref{eq:strongconvF} and a combination of the weak convergence of $\Omega_{\ell,D^c}$ and the strong convergence of $F_\ell$ implies
\[
 \lim_{\ell\to\infty} \int_\R \Omega_{\ell, D^c}(x,y) F_\ell(x,y) \frac{\dx \dy}{|x-y|} = \int_\R \Omega_{D^c}(x,y) F(x,y) \frac{\dx \dy}{|x-y|}.
\]
This establishes \eqref{eq:number3}.

\underline{\textsc{Step 5.}} We claim that
\begin{equation}\label{eq:number4}
 \lim_{\ell \to \infty} f_\ell [A_\ell \psi] = f[A\psi].
\end{equation}
Indeed, this holds because $A_\ell \psi$ is uniformly bounded in $\change{\dot{W}}^{\frac12,2}$\change{, $A_\ell \psi$ converges weakly to $A\psi$ in $\dot{W}^{\frac12,2}$,} and by assumption $f_\ell \to f$ in ${{W}}^{-\frac12,2}$.

\underline{\textsc{Step 6.}} Passing to the limit. 

Passing with $\ell\to\infty$ in \eqref{eq:equationbeforepassingtothelimit}, using \eqref{eq:limitI}, \eqref{eq:secondtermdidIusethisname}, \eqref{eq:number3}, and \eqref{eq:number4}, we obtain 
\begin{equation}\label{eq:almostperfect}
 \int_{\R} A d_\frac12 u \cdot d_\frac12 \psi= \int_{\R} \brac{\Omega_{D}^{A}}^\ast\cdot u d_{\frac{1}{2}}\psi+\int_{\R} A \Omega_{D^c}\cdot d_{\frac{1}{2}} u \, \psi+ f[A \psi].
\end{equation}
By \eqref{eq:secondtermdidIusethisname} we know that $\brac{\Omega_{D}^{A}}^\ast$ is $\frac12$-divergence free and thus by \Cref{la:prodrule} we have
\[
 \int_{\R} \brac{\Omega_{D}^{A}}^\ast\cdot u d_{\frac{1}{2}}\psi = \int_{\R} \Omega_{D}^{A}\cdot d_\frac12 u\psi,
\]
which combined with \eqref{eq:almostperfect} and formulas $\Omega_D^A = A\Omega_D - d_\frac12A$ and $\Omega_{D^c}=\Omega-\Omega_D$ gives
\begin{equation}\label{eq:thisoneisperfect}
 \int_{\R} A d_\frac12 u \cdot d_\frac12 \psi= \int_{\R} \Omega^{A}\cdot d_{\frac{1}{2}}u\psi+ f[A \psi].
\end{equation}
This holds for any $\psi \in C_c^\infty(D_1)$. By density \change{we can invoke} \Cref{le:reformulateOmegatoA}, which leads to the claim.
\end{proof}

\begin{corollary}\label{co:weskconvexceptSigma}
Let $u_\ell$, $\Omega_\ell$, and $f_\ell$ be as in \Cref{th:weakconvergence}. Let $D\subset \R$. Then there exits a locally finite $\Sigma \subset D$ such that
\[
 (-\Delta)^{\frac12} u = \Omega \cdot d_\frac12 u +f\quad \text{in } D \setminus \Sigma.
\]
\end{corollary}
\begin{proof}
We follow in spirit the covering argument of Sacks--Uhlenbeck \cite[Proposition 4.3 \& Theorem 4.4]{Sucks1}.

By assumptions there is a number $\Lambda>0$ such that $\sup_{\ell\in\N} \|
\Omega_\ell\|_{L^2(\Ep_{od}^1\R)}<\Lambda$.

Let $\alpha\in\N$ and \change{let} $\mathcal {B}_\alpha \coloneqq \{B(x_{i,\alpha},2^{-\alpha})\colon x_{i,\alpha}\in D\}$ be a family of balls such that $D\subset \bigcup \mathcal {B}_\alpha$ and each point $x\in D$ is covered at most $\lambda$ times, and such that for a smaller radius we still have $D \subset \bigcup_{i} B(x_{i,\alpha}, 2^{-\alpha-1})$. Then
\[
 \sum_{i} \int_{B(x_{i,\alpha},2^{-\alpha})}\int_{\R} |\Omega_\ell(x,y)|^2 \frac{\dif x \dif y}{|x-y|} < \Lambda \lambda.
\]
Now, let $\sigma>0$ be the number from \Cref{th:main}, then there exists at most $\frac{\Lambda \lambda}{\sigma}$ balls in $\mathcal {B}_\alpha$ on which
\[
 \int_{B(x_{i,\alpha},2^{-\alpha})}\int_{\R} |\Omega_\ell(x,y)|^2 \frac{\dif x \dif y}{|x-y|} > \sigma.
\]
Thus, by \Cref{pr:weakconv:1}, we obtain that except for $K<\frac{\Lambda \lambda}{\sigma} +1$ balls from $\mathcal {B}_\alpha$ we have
\begin{equation}\label{eq:onballs}
 \int_\R d_\frac12 u \cdot d_\frac12 \vp_i = \int_\R \Omega \cdot d_\frac12 u \vp_i + f[\vp_i]\quad \text{ for all } \vp_i\in C_c^\infty(B(x_{i,\alpha},2^{-\alpha-1})).
\end{equation}
Let us denote those balls by $B(y_{i,\alpha}, 2^{-\alpha})$ for $i=1,\dots, K$. Then by \eqref{eq:onballs} we get
 \begin{equation}\label{eq:onpsioki}
  \int_\R d_\frac12 u \cdot d_\frac12 \psi = \int_\R \Omega \cdot d_\frac12 u \psi + f[\psi], \quad \text{ for all }\psi \in C_c^\infty(D\setminus \bigcup_{i\le K} \change{\overline{B}}(y_{i,\alpha},2^{-\alpha-1})).
 \end{equation}
Since $\bigcup_{\alpha\in\N}\brac{D\setminus \bigcup_{i=1}^K \change{\overline{B}}(y_{i,\alpha},2^{-\alpha-1})} = D\setminus \{x_1,\ldots,x_K\}$\change{, \eqref{eq:onpsioki} holds for any $\psi\in C_c^\infty(D\setminus \Sigma)$, where $\Sigma\coloneqq \{x_1,\ldots,x_K\}$. This gives the claim}.
\end{proof}

In order to conclude we will need a removability of singularities lemma, compare with \cite[Proposition 4.7]{MS20}.

\begin{lemma}\label{le:removability}
Let $u \in \dot{W}^{\frac12,2}(\R,\R^N)$, $f \in L^1(\R,\R^N)$, and $g \in W^{-\frac{1}{2},2}(\R)$.  Assume that for some locally finite set $\Sigma \subset D$ we have 
\[
 (-\Delta)^\frac12 u = f +g \quad \text{in }D \setminus \Sigma.
\]
Then
\[
 (-\Delta)^\frac12 u = f +g\quad \text{in }D.
\]
\end{lemma}

\begin{proof}
For simplicity of presentation let us assume that $\Sigma=\{x_0\}$. By definition we have for any $\vp\in C^\infty_c(D\setminus \{x_0\})$
\[
 \int_D \int_D \frac{(u(x) - u(y))(\vp(x) - \vp(y))}{|x-y|^{2}} \dif x \dif y = \int_D f(x) \vp(x) \dif x+g[\vp]. 
\]
Let $\{\zeta_\ell\}_{\ell \in \N} \subset C_c^\infty(D,[0,1])$ be the sequence from \Cref{le:zerocapacityfunctions}, i.e., such that for all $\ell\in\N$ we have
\begin{equation}\label{eq:zetaproperties}
 \zeta_\ell \equiv 1 \text{ on } B_{\rho_\ell}(x_0), \quad \zeta_\ell \equiv 0 \text{ outside } B_{R_\ell}(x_0), \quad \text{ and }\lim_{\ell\to\infty}[\zeta_\ell]_{W^{\frac12,2}(D)}=0
\end{equation}
for a $0<\rho_\ell<R_\ell\to 0$ as $\ell\to\infty$.

Now let $\psi \in C^\infty_c(D)$ and then $\psi_\ell \coloneqq \psi(1-\zeta_\ell)\in C_c^\infty(\Sigma\setminus \{x_0\})$ is an admissible test function and we have
\begin{equation}\label{eq:eqwithell}
\int_D \int_D \frac{(u(x) - u(y))(\psi(x)-\psi(y))}{|x-y|^{2}} \dif x \dif y - \mathcal{I}_\ell
= \int_D f(x) \psi(x)\dif x + g[\psi] - \mathcal{II}_\ell - \mathcal{III}_{\ell}.
  \end{equation}  
We have
\begin{equation}\label{eq:zetaellfirstterm}
\begin{split}
\mathcal{I}_\ell &\coloneqq \int_D \int_D \frac{(u(x) - u(y))(\psi(x)\zeta_\ell(x) - \psi(y)\zeta_\ell(y))}{|x-y|^{2}} \dif x \dif y\\
&= \int_D \int_D \frac{(u(x) - u(y))\psi(x)(\zeta_\ell(x) - \zeta_\ell(y))}{|x-y|^{2}} \dif x \dif y\\
&\quad + \int_D \int_D \frac{(u(x) - u(y))(\psi(x)-\psi(y))\zeta_\ell(y)}{|x-y|^{2}} \dif x \dif y\\
&\le \|\psi\|_{L^\infty(D)} [u]_{W^{\frac12,2}(D)}[\zeta_\ell]_{W^{\frac12,2}(D)}+ \int_{B_{R_\ell}}\int_{D}\frac{\change{|}u(x) - u(y)\change{||}\psi(x)-\psi(y)\change{|}}{|x-y|^{2}} \dif x \dif y\change{.} 
\end{split}
\end{equation}
Thus, by \eqref{eq:zetaproperties} and by the absolute continuity of the integral we have $\lim_{\ell\to\infty} \mathcal{I}_\ell=0$.

Secondly,
\begin{equation}\label{eq:zetaellsecondterm}
\mathcal{II}_\ell \coloneqq \int_D f(x) \psi(x)\zeta_\ell(x) \dif x\le \|\psi\|_{L^\infty}\int_{B_{R_\ell}} |f(x)| \dif x \xrightarrow{\ell\to\infty} 0, 
\end{equation}
by the absolute continuity of the integral.

Thus, passing with $\ell \to \infty$ in \eqref{eq:eqwithell} we get for any $\psi\in C_c^\infty(D)$
\[
 \int_D \int_D \frac{(u(x) - u(y))(\psi(x)-\psi(y))}{|x-y|^{2}} \dif x \dif y
= \int_D f(x) \psi(x)\dif x.
\]
Lastly,
\[
 \mathcal{III}_\ell \coloneqq g[\psi\,\zeta_\ell] \xrightarrow{\ell \to \infty} 0,
\]
because, by \eqref{eq:zetaproperties}, we have $[\psi\,\zeta_\ell]_{W^{\frac12,2}} \xrightarrow{\ell \to \infty} 0$.

This finishes the proof.

\end{proof}

\begin{proof}[Proof of \Cref{th:weakconvergence}]
Combining \Cref{co:weskconvexceptSigma} and \Cref{le:removability} we obtain the claim.
\end{proof}

\appendix

\section{Nonlocal Hodge decomposition}
\begin{lemma}\label{la:hodge}
Let $p>1$, $s\in(0,1)$, $G \in L^p(\Ep_{od}^1 \R^n)$ then there exists a \change{decomposition}\footnote{\change{The decomposition is unique if we normalize $a$}}
\[
 G = d_{s} a + B,
\]
where $a \in \dot{W}^{s,p}(\R^n)$ and $B \in L^p(\Ep_{od} ^1\R^n)$ with $\div_{s} B = 0$. Moreover,
\begin{equation}\label{eq:hodgeinequality}
 \|B\|_{L^p(\Ep_{od}^1 \R^n)} + [a]_{W^{s,p}(\R^n)} \aleq \|G\|_{L^p(\Ep_{od}^1\R^n)}.
\end{equation}
\end{lemma}
\begin{proof}
Since $G\in L^p(\Ep_{od}^1 \R^n)$ we have $\div_s G \in \brac{W^{s,p'}(\R^n)}^\ast$, namely
\[
 \div_s G[\varphi] \aleq \|G\|_{L^p(\Ep_{od}^1\R^n)}\, [\varphi]_{W^{s,p'}(\R^n)}.
\]
\change{Recall that for $0<s<1$ and $1\le p <\infty$ we have $\dot{W}^{s,p}(\R^n) = \dot{F}^s_{p,p}(\R^n)$ \cite[2.3.5]{1983functionspaces}. Moreover, } $\div_s G \in \dot{F}^{-s}_{p,p}$, since $(-\Delta)^{-s}: \dot{F}^{\change{s}}_{p,p}(\R^n) \to \dot{F}^{\change{-s}}_{p,p}(\R^n)$ is an isomorphism \cite[\textsection 2.6.2, Proposition 2, p.95]{RS96}. In particular, there is a unique unique solution $a \in \dot{F}^s_{p,p}(\R^n)$ to the distributional equation
\[
 (-\Delta)^s a = \div_{s} G.
\]
with 
\[
 [a]_{\dot{F}^s_{p,p}(\R^n)} \aleq [\div_s G]_{F^{-s}_{p',p'}(\R^n)} \aleq \|G\|_{L^p(\Ep_{od}^1 \R^n)}.
\]
We have found $a \in \dot{F}^{s}_{p,p}(\R^n) = \dot{W}^{s,p}(\R^n)$, and we have 
\[
 \int_{\R^n} d_s a \cdot d_s \vp = \int_{\R^n} F \vp \quad \forall \vp \in C_c^\infty(\R^n).
\]
The uniqueness of $a$ up to a normalization assumption would follow by considering a difference of two solutions and an application of nonlocal Liouville theorem \cite[Theorem 1.1]{Fall-2016}. 

Now define $B\coloneqq G - d_s a$. We have
\[
 \div_s B = \div_s G - \div_s (d_s a) = \div_s G - (-\Delta)^s a =0,
\]
which finishes the proof.
\end{proof}
\section{Localization}

The next Proposition follows from a relatively straight-forward localization results, see, e.g., \cite{MSY20}.
\begin{proposition}\label{pr:localization}
Assume $D_1 \subset \subset D_2 \subset \subset D' \subseteq D \subseteq \R$ open intervals and let $u \in L^1(\R,\R^N) + L^\infty(\R,\R^N)\cap \dot{W}^{\change{\frac 12},2}(D,\R^N)$ be a solution to
\[
 \laps{1}_{D} u = \Omega \cdot_{D} d_{\frac{1}{2}} u + f \quad \text{in } D'.
\]
That is, assume 
\begin{equation}\label{eq:eqonD}
\begin{split}
 \int_{D} \int_{D} &\frac{(u(x)-u(y)) (\varphi(x)-\varphi(y))}{|x-y|^{2}}{\dx \dy}\\
 &\quad = \int_{D} \int_{D} \Omega(x,y) d_{\frac{1}{2}} u(x,y) \varphi(x)\, \frac{\dx \dy}{|x-y|} + \int_{D} f \varphi, \quad \forall \varphi \in C_c^\infty(D').
 \end{split}
\end{equation}

Let $\eta \in C_c^\infty(D_1)$ and set $v \coloneqq \eta u$ and $\tilde{\Omega}_{ij}(x,y) = \chi_{D_2}(x)\chi_{D_2}(y)\Omega_{ij}(x,y)$. Then
\[
 \laps{1} v = \tilde{\Omega} \cdot d_{\frac{1}{2}} v + \eta f + \mathcal{G}(u,\cdot) \quad \text{in }\R,
\]
where $\mathcal{G}$ is a bilinear form with the following estimates for any $s \in (0,\frac{1}{2})$ and $\eps > 0$
\[
\begin{split}
 |\mathcal{G}(u,\change{\vp})| \leq C(\eta,s,\eps,D_1,D_2)&\, \brac{1+\|\Omega\|_{L^2(\Ep^1_{od} D)}}\\
 \quad &\cdot \brac{\|u\|_{L^2(D)+L^\infty(D)} + [u]_{W^{s,2}(D_2)}}\\
 \quad &\cdot \brac{\|\varphi\|_{L^2(D)+L^\infty(D)} + \|\varphi\|_{L^{\frac{1}{s}}(D_2)} + \|\varphi\|_{L^1 + L^\infty(\R)} + [\varphi]_{W^{\eps,\frac{2}{2s+1}}(D_2)}}.
 \end{split}
\]
In particular we have
\[
 \|\tilde{\Omega}\|_{L^2(\Ep^1_{od} \R)} \leq \|\Omega\|_{L^2(\Ep^1_{od} D_2)}.
\]
\end{proposition}
\begin{proof}
Let $\vp\in C_c^\infty(\R)$. We have 
\[
\begin{split}
 &\brac{\eta(x) u(x) - \eta(y) u(y)} \brac{\varphi(x) -\varphi(y)}\\
 &\quad=(u(x)-u(y)) (\eta(x) \varphi(x) - \eta(y) \varphi(y)) +(\eta(x)-\eta(y)) \brac{u(y) \varphi(x) - u(x) \varphi(y)}.
 \end{split}
\]
Since $\eta \varphi \in C_c^\infty(D')$ it is an admissible test function and we have from the equation \eqref{eq:eqonD}
\[
\begin{split}
 &\int_{D} \int_{D} \frac{(v(x)-v(y)) (\varphi(x)-\varphi(y))}{|x-y|^{2}}\dx \dy\\
 &\quad = \int_{D} \int_{D} \Omega(x,y) d_{\frac{1}{2}} u(x,y) \eta(x)\varphi(x)\, \frac{\dx \dy}{|x-y|} + \int_{\R} f \eta \varphi +\mathcal{G}_1(u,\varphi).
 \end{split}
\]
Here, 
\[
 \mathcal{G}_1(u,\varphi) = \int_{D} \int_{D} \frac{(\eta(x)-\eta(y)) \brac{u(y) \varphi(x) - u(x) \varphi(y)}}{|x-y|^{2}}{\dx \dy}.
\]
Moreover, we have 
\[
\begin{split}
 &\int_{\R} \int_{\R} \frac{(v(x)-v(y)) (\varphi(x)-\varphi(y))}{|x-y|^{2}}{\dx \dy}\\
 &\quad = \int_{D} \int_{D} \frac{(v(x)-v(y)) (\varphi(x)-\varphi(y))}{|x-y|^{2}}{\dx \dy} + \mathcal{G}_2(u,\varphi),
\end{split}
 \]
where, because $\supp v \subset D_1$,
\[
\begin{split}
\mathcal{G}_2(u,\varphi) = 2\int_{D_1} v(x) \int_{\R\setminus D} \frac{ (\varphi(x)-\varphi(y))}{|x-y|^{2}}{\change{ \dy \dx}}.
\end{split}
\]
That is we have
\[
\begin{split}
 &\int_{\R} \int_{\R} \frac{(v(x)-v(y)) (\varphi(x)-\varphi(y))}{|x-y|^{2}}{\dx \dy}\\
 &\quad= \int_{D} \int_{D} \Omega(x,y) d_{\frac{1}{2}} u(x,y) \eta(x)\varphi(x)\, \frac{\dx \dy}{|x-y|} + \int_{\R} f \eta \varphi
 +\mathcal{G}_1(u,\varphi)+\mathcal{G}_2(u,\varphi).
 \end{split}
\]
Furthermore, since 
\[
 d_{\frac{1}{2}}u(x,y) \eta(x) = d_{\frac{1}{2}} (\eta u)(x,y) - u(y) d_{\frac{1}{2}}\eta(x,y)
\]
and $\supp v\subset D_1$, we have
\[
\begin{split}
  &\int_{D} \int_{D} \Omega(x,y) d_{\frac{1}{2}} u(x,y) \eta(x)\varphi(x)\, \frac{\dx \dy}{|x-y|}\\
  &=\int_{D} \int_{D} \Omega(x,y) d_{\frac{1}{2}} v(x,y)\, \varphi(x)\, \frac{\dx \dy}{|x-y|}
  -\int_{D} \int_{D} \Omega(x,y) u(y) d_{\frac{1}{2}}\eta(x,y)\varphi(x)\, \frac{\dx \dy}{|x-y|}\\
&=\int_{\R} \int_{\R} \chi_{D_2}(x) \chi_{D_2}(y)\Omega(x,y) d_{\frac{1}{2}} v(x,y)\, \varphi(x)\, \frac{\dx \dy}{|x-y|}\\
&\quad+\int_{D \setminus D_2} \int_{D_2} \Omega(x,y) d_{\frac{1}{2}} v(x,y)\, \varphi(x)\, \frac{\dx \dy}{|x-y|}\\
&\quad+\int_{D_2} \int_{D \setminus D_2} \Omega(x,y) d_{\frac{1}{2}} v(x,y)\, \varphi(x)\, \frac{\dx \dy}{|x-y|}\\
  &\quad-\int_{D} \int_{D} \Omega(x,y) u(y) d_{\frac{1}{2}}\eta(x,y)\varphi(x)\, \frac{\dx \dy}{|x-y|}.
  \end{split}
\]
So if we set 
\[
 \mathcal{G}_3(u,\varphi) \coloneqq \int_{D \setminus D_2} \int_{D_2} \Omega(x,y) d_{\frac{1}{2}} v(x,y)\, \varphi(x)\, \frac{\dx \dy}{|x-y|} + \int_{D_2} \int_{D \setminus D_2} \Omega(x,y) d_{\frac{1}{2}} v(x,y)\, \varphi(x)\, \frac{\dx \dy}{|x-y|}\\
\]
and
\[
 \mathcal{G}_4(u,\vp) \coloneqq -\int_{D} \int_{D} \Omega(x,y) u(y) d_{\frac{1}{2}}\eta(x,y)\varphi(x)\, \frac{\dx \dy}{|x-y|},
\]
then we have shown for any $\varphi \in C_c^\infty(\R)$,
\[
\begin{split}
 \int_{\R} \int_{\R} \frac{(v(x)-v(y)) (\varphi(x)-\varphi(y))}{|x-y|^{2}}{\dx \dy}
 = \int_{\R} \tilde{\Omega}\cdot d_{\frac{1}{2}} v \, \varphi + \int_{\R} f \eta \varphi+\sum_{i=1}^4\mathcal{G}_i(u,\varphi).
 \end{split}
\]
It remains to estimate each $\mathcal{G}_i(u,\varphi)$.

\underline{Estimate of $\mathcal{G}_1$:}
By the support of $\eta$ we have
\begin{equation}\label{eq:G1split}
\begin{split}
 \mathcal {G}_1(u,\vp) 
 &= \int_{D_2} \int_{D_2} \frac{(\eta(x)-\eta(y)) \brac{u(y) \varphi(x) - u(x) \varphi(y)}}{|x-y|^{2}}{\dx \dy}\\
 &\quad+ 2 \int_{D_1}\int_{D\setminus D_2}\frac{(\eta(x)-\eta(y)) \brac{u(y) \varphi(x) - u(x) \varphi(y)}}{|x-y|^{2}}{\dx \dy}.
 \end{split}
\end{equation}

As for the first term we have
\begin{equation}\label{eq:firstermfirst}
\begin{split}
 &\int_{D_2} \int_{D_2} \frac{(\eta(x)-\eta(y)) \brac{u(y) \varphi(x) - u(x) \varphi(y)}}{|x-y|^{2}}{\dx \dy}\\
 &\leq \|\eta\|_{\lip} \int_{D_2} \int_{D_2} \frac{|u(y) \varphi(x)-u(x)\varphi(y)|}{|x-y|}{\dx \dy}\\
 &\aleq\|\eta\|_{\lip} \brac{\int_{D_2} |u(y)|\int_{D_2} \frac{|\varphi(x)-\varphi(y)|}{|x-y|}{\dx \dy}+ \int_{D_2} |\varphi(y)|\int_{D_2} \frac{|u(x)-u(y)|}{|x-y|}{\dx \dy}}\\
 &\aleq\|\eta\|_{\lip} \int_{D_2} |u(y)-(u)_{D_2}|\int_{D_2} \frac{|\varphi(x)-\varphi(y)|}{|x-y|}{\dx \dy}\\
 &\quad +\|\eta\|_{\lip}\|u\|_{L^1(D_2)} \int_{D_2} \int_{D_2} \frac{|\varphi(x)-\varphi(y)|}{|x-y|}{\dx \dy}\\
 &\quad+\|\eta\|_{\lip} \int_{D_2} |\varphi(y)|\int_{D_2} \frac{|u(x)-u(y)|}{|x-y|}{\dx \dy}.
 \end{split}
\end{equation}
We observe that for any $p \in (1,\infty)$ and any $\eps > 0$ we have
\[
\begin{split}
 \int_{D_2}\brac{\int_{D_2} \frac{|\varphi(x)-\varphi(y)|}{|x-y|} \dx}^p \dy
 &=\int_{D_2}\brac{\int_{D_2} \frac{|\varphi(x)-\varphi(y)|}{|x-y|^\eps}  |x-y|^{\eps}\frac{\dx}{|x-y|} }^p \dy\\
 &\aleq [\varphi]_{W^{\eps,p}(D_2)}\,\change{\sup_{y\in D_2}} {\brac{\int_{D_2} |x-y|^{\eps p'} \frac{\dx}{|x-y|} }^{\frac{p}{p'}} }\\
 &\aleq C(D_2)[\varphi]_{W^{\eps,p}(D_2)}.
 \end{split}
\]
Thus, for any $\eps>0$ and any $s\in(0,\frac12)$ we have
\begin{equation}\label{eq:firsttermsecond}
\begin{split}
&\int_{D_2} |u(y)-(u)_{D_2}|\int_{D_2} \frac{|\varphi(x)-\varphi(y)|}{|x-y|}{\dx \dy} + \|u\|_{L^1(D_2)} \int_{D_2} \int_{D_2} \frac{|\varphi(x)-\varphi(y)|}{|x-y|}{\dx \dy}\\
&\aleq C(D_2)\brac{\|u-(u)_{D_2}\|_{L^{\frac{2}{1-2s}}(D_2)}[\varphi]_{W^{\eps,\frac{2}{2s+1}}(D_2)} +  \|u\|_{L^1(D_2)} [\varphi]_{W^{\eps,\frac{2}{2s+1}}(D_2)}}.
\end{split}
\end{equation}
We also have
\begin{equation}\label{eq:firsttermthird}
 \int_{D_2} |\varphi(y)|\int_{D_2} \frac{|u(x)-u(y)|}{|x-y|}{\dx \dy} \aleq \|\varphi\|_{L^2(D_2)}\, [u]_{W^{s,2}(D_2)}. 
\end{equation}
Combining \eqref{eq:firstermfirst} with \eqref{eq:firsttermsecond} \change{(in which we use Poincar\`{e} inequality)} and \eqref{eq:firsttermthird} we obtain
\begin{equation}\label{eq:firstermfinal}
\begin{split}
 &\int_{D_2} \int_{D_2} \frac{(\eta(x)-\eta(y)) \brac{u(y) \varphi(x) - u(x) \varphi(y)}}{|x-y|^{2}}{\dx \dy}\\
 &\quad \aleq \|\eta\|_{\lip}\brac{\|u\|_{L^1(D_2)} + [u]_{W^{s,2}(D_2)}}\, \brac{\|\varphi\|_{L^2(D_2)} + [\varphi]_{W^{\eps,\frac{2}{2s+1}}(D_2)}}.
\end{split}
 \end{equation}

For the second term of \eqref{eq:G1split} observe that for $x \in D_1$ and $y \in D \setminus D_2$ we have $|x-y| \aeq 1 + |y|$, so we have 
\begin{equation}\label{eq:G1estimatesecondterm}
\begin{split}
 &2\abs{\int_{D_1} \int_{D \setminus D_2} \frac{(\eta(x)-\eta(y)) \brac{u(y) \varphi(x) - u(x) \varphi(y)}}{|x-y|^{2}}\change{\dy \dx}}\\
 &\quad\aleq\|\eta\|_{L^\infty}\abs{\int_{D_1} \int_{D \setminus D_2} \frac{\abs{u(y)} \abs{\varphi(x)} + \abs{u(x)} \abs{\varphi(y)} }{1+|y|^{2}}\change{\dy \dx }}\\
 &\quad\aleq \|\eta\|_{L^\infty}\|u\|_{L^1+L^\infty(D)}\, \|\varphi\|_{L^1+L^\infty(D)}.
 \end{split}
\end{equation}

Thus, by \eqref{eq:G1split}, \eqref{eq:firstermfinal}, and \eqref{eq:G1estimatesecondterm} we get
\begin{equation}\label{eq:G1final}
 |\mathcal{G}_1(u,\vp)| \aleq \brac{\|u\|_{L^1+L^\infty(D)}+[u]_{W^{s,2}(D_2)}}\brac{\|\varphi\|_{L^2+L^\infty(D)} + [\varphi]_{W^{\eps,\frac{2}{2s+1}}(D_2)}}.
\end{equation}

\underline{Estimate of $\mathcal{G}_2$:}
Similarly as in \eqref{eq:G1estimatesecondterm}, if $x \in D_1$ and $y \in \R \setminus D$ we have $|x-y| \aeq 1+|y|$, and thus
\[
|\mathcal{G}_2(u,\varphi)| \aleq \|\eta u\|_{L^2(D)}\brac{\|\varphi\|_{L^2(\change{D_1})} + \|\varphi\|_{L^1\change{+L^\infty}(\R)}} \aleq \|u\|_{L^2(D_1)}\brac{\|\varphi\|_{\change{L^2+L^\infty}(D)} + \|\varphi\|_{L^1 + L^\infty(\R)}}.
\]
\underline{Estimate of $\mathcal{G}_3$:}
Using the support of $v$, observing again that $|x-y| \ageq 1+|y|$ if $y \in \R \setminus D_2$ and $x \in D_1$ \change{ we get}
\[
\begin{split}
 &|\mathcal{G}_3(u,\varphi)|\\ 
 &\aleq \|\Omega\|_{L^2(\Ep^1_{od} D)}
  \brac{
 \int_{D \setminus D_2} \int_{D_1} |u(x)|^2\, |\varphi(x)|^2\, \frac{\dx \dy}{1+|y|^{2}} + \int_{D_1} \int_{D \setminus D_2} |u(y)|^2 |\varphi(x)|^2\, \frac{\dx \dy}{1+|x|^{2}}}^{\frac{1}{2}}\\
 &\aleq\|\Omega\|_{L^2(\Ep^1_{od} D)}\, \brac{\|u\varphi\|_{L^2(D_1)} + \|\varphi\|_{L^2\change{+L^\infty}(D)}\, \|u\|_{L^2(D_1)}}\\
 &\aleq\|\Omega\|_{L^2(\Ep^1_{od} D)}\, \brac{\|u\|_{L^1(D_1)} \|\varphi\|_{L^2(D_1)} + 
 \|u-(u)_{D_1}\|_{L^{\frac{2}{1-2s}}(D_1)}\, \|\varphi\|_{L^{\frac{1}{s}}(D_1)}
 + \|\varphi\|_{L^2\change{+L^\infty}(D)}\, \|u\|_{L^2(D_1)}}\\
 &\aleq\|\Omega\|_{L^2(\Ep^1_{od} D)}\, \brac{\|u\|_{L^1(D_1)} \|\varphi\|_{L^2(D_1)} + 
 [u]_{W^{s,2}(D_1)}\, \|\varphi\|_{L^{\frac{1}{s}}(D_1)}
 + \|\varphi\|_{L^2\change{+L^\infty}(D)}\, \|u\|_{L^2(D_1)}}\\
 &\aleq\|\Omega\|_{L^2(\Ep^1_{od} D)}\, \brac{\|u\|_{L^2(D_1)} + 
 [u]_{W^{s,2}(D_1)}} \brac{\|\varphi\|_{L^{\frac{1}{s}}(D_1)}
 + \|\varphi\|_{L^2\change{+L^\infty}(D)}}.
 \end{split}
\]
This argument works for any $s \in (0,\frac{1}{2})$.

\underline{Estimate of $\mathcal{G}_4$:}
We have 
\[
 |\mathcal{G}_4(u,\varphi)| \aleq \|\Omega\|_{L^2(\Ep^1_{od} D)} \brac{\int_{D} \int_{D} |u(y) d_{\frac{1}{2}}\eta(x,y)\varphi(x)|^2\, \frac{\dx \dy}{|x-y|}}^{\frac{1}{2}}.
\]
Now observe that $|d_{\frac{1}{2}} \eta(x,y)|^2 \leq \|\eta\|_{\lip}^2 |x-y|$, thus 
\[
 \brac{\int_{D} \int_{D} |u(y) d_{\frac{1}{2}}\eta(x,y)\varphi(x)|^2\, \frac{\dx \dy}{|x-y|}}^{\frac{1}{2}} \aleq \|u\|_{L^2(D)}\, \|\varphi\|_{L^2(D)}.
\]
On the other hand 
\[
 \brac{\int_{D} \int_{D} |u(y) d_{\frac{1}{2}}\eta(x,y)\varphi(x)|^2\, \frac{\dx \dy}{|x-y|}}^{\frac{1}{2}} \aleq [\eta]_{W^{\frac{1}{2},2}}\|u\|_{L^\infty(D)}\, \|\varphi\|_{L^\infty(D)}.
\]
We also have 
\[
 \brac{\int_{D} \int_{D} |u(y) d_{\frac{1}{2}}\eta(x,y)\varphi(x)|^2\, \frac{\dx \dy}{|x-y|}}^{\frac{1}{2}} \aleq \|u\|_{L^\infty(D)}\, \|\varphi\|_{L^2(D)}\, \sup_{x\in D} \brac{\int_{D} \frac{|\eta(x)-\eta(y)|^2}{|x-y|^2}\dy }^{\frac{1}{2}}
\]
and
\[
 \sup_{x\in D} \brac{\int_{D} \frac{|\eta(x)-\eta(y)|^2}{|x-y|^2}\dy }^{\frac{1}{2}} \aleq \|\eta\|_{\lip}.
\]
Thus combining the estimates on $\mathcal{G}_4$ we obtain
\[
 |\mathcal{G}_4(u,\varphi)| \aleq \|\Omega\|_{L^2(\Ep^1_{od} D)}  \|u\|_{L^2+L^\infty(D)} \|\varphi\|_{L^2+L^\infty(D)}.
\]

\end{proof}

\section{A Sobolev inequality}\label{s:weirdsob}
\begin{theorem}\label{th:fspqchar}
Let $s \in (0,1)$, $p,q \in (1,\infty)$ and $f \in L^p(\R^n)$ then
\begin{enumerate}
\item \[
         [f]_{\dot{F}^s_{p,q}(\R^n)} \aleq [f]_{W^{s}_{p,q}(\R^n)};
      \]
 \item if $p > \frac{nq}{n+sq}$ then
    \[
        [f]_{W^{s}_{p,q}(\R^n)} \aleq [f]_{\dot{F}^s_{p,q}(\R^n)}.
       \]
\end{enumerate}
The constants depend on $s,p,q,n$ and are otherwise uniform.
\end{theorem}
While characterizations such as \Cref{th:fspqchar} are well-known for Besov spaces, for Triebel spaces this seems to have been known only for $q=p$ (where it follows from the Besov-space characterization), $q=2$ where it is a result due to Stein and Fefferman, \cite{Stein1961,F70}. It was also known ``for large s'' \cite[Section 2.5.10]{1983functionspaces}. \change{Although a conjecture that} \Cref{th:fspqchar} holds is very natural, quite surprisingly, to the best of our knowledge, the first time \Cref{th:fspqchar} has been proven was recently by \change{Prats and Saksman \cite[Theorem 1.2]{Prats-Saksman} (see also \cite{P19} for further development)}, \change{but see also \cite{S89,T89}}.

\begin{corollary}\label{co:DsSobinequ}
Let $s \in (0,1)$, $t \in (s,1)$ and $p,p^\ast \in (1,\infty)$ where 
\begin{equation}\label{eq:constantslabel}
 s-\frac{n}{p^\ast} = t-\frac{n}{p}.
\end{equation}
If $q \in (1,\infty)$ such that $p^\ast > \frac{nq}{n+sq}$ we have
\[
 \||\mathcal{D}_{s,q} f|\|_{L^{p^\ast}(\R^n)} \aleq \|\laps{t} f\|_{L^p(\R^n)}.
\]
More precisely, in terms of Lorentz spaces we have for any $r \in [1,\infty]$,
\[
 \||\mathcal{D}_{s,q} f|\|_{L^{(p^\ast,r)}(\R^n)} \aleq \|\laps{t} f\|_{L^{(p,r)}(\R^n)}.
\]
\end{corollary}
\begin{proof}
From \Cref{th:fspqchar} we have 
\[
  \||\mathcal{D}_{s,q} f|\|_{L^{p^\ast}(\R^n)} \aeq [f]_{F^s_{p^\ast,q}(\R^n)}.
\]
We recall the Sobolev-embedding theorem for Triebel-Lizorkin spaces $\dot{F}^t_{p,\tilde{q}} \hookrightarrow \dot{F}^{\change{s}}_{\change{p^\ast},q}$ for any $q,\tilde{q} \in (1,\infty)$ 
\change{ and $s,\, t,\, p,\,p^\ast$ satisfying \eqref{eq:constantslabel} (see, e.g., \cite[Theorem 2.7.1 (ii)]{1983functionspaces})}. Thus,
\[
  \||\mathcal{D}_{s,q} f|\|_{L^{p^\ast}(\R^n)} \aleq [f]_{F^t_{p,2}(\R^n)} \aleq \|\laps{t} f\|_{L^{p}(\R^n)}.
\]
As for the Lorentz-space estimate we can argue by real interpolation. Indeed, fix $s,q,p,p^\ast$. Observe that $f \mapsto |\mathcal{D}_{s,q} f|$ is a sublinear operator.We can find $p_1 < p < p_2$ such that $p_1$ and $p_2$ are still admissible, and thus we have 
\[
  \||\mathcal{D}_{s,q} f|\|_{L^{p_i^\ast}(\R^n)} \aleq [f]_{F^t_{p,2}(\R^n)} \aleq \|\laps{t} f\|_{L^{p_i}(\R^n)} \quad i=1,2.
\]
From real interpolation we now obtain the Lorentz space claim.
\end{proof}

\section{A sequence of cut-off functions in the critical Sobolev space}\label{sc:0capacity}
For readers convenience we present here a proof of a well known result, which essentially says that in the critical Sobolev space a point has zero capacity. See for example \cite[Theorem 5.1.9]{Adams-Hedberg}, compare also with a similar construction \cite[Lemma 3.2]{Monteil-VanSchaftingen}.

\begin{lemma}\label{le:zerocapacityfunctions}
There exists a sequence of functions with the following properties:

$\{\zeta_\ell\}_{\ell \in \N} \subset C_c^\infty(\R,[0,1])$ and for all $\ell\in\N$ we have
\begin{equation}\label{eq:zetaproperties2}
 \zeta_\ell \equiv 1 \text{ on } B_{\rho_\ell}(x_0), \quad \zeta_\ell \equiv 0 \text{ outside } B_{R_\ell}(x_0), \quad \text{ and }\lim_{\ell\to\infty}[\zeta_\ell]_{W^{\frac12,2}(\R)}=0
\end{equation}
for a sequence of radii $0<\rho_\ell<R_\ell\to 0$ as $\ell\to\infty$. 
\end{lemma}

\begin{proof}
Let $f(x) = \log\log \brac{1 + \frac{1}{|x|^2}}\in W^{1,2}(B^2_1,\R)$ be an unbounded function. We define
\[
 \tilde{Z}_k(x) \coloneqq \left\{
 \begin{array}{ll}
  1 & \text{ if } f(x)\ge k+1,\\
  f(x) - k & \text{ if } k\le f(x)\le k+1,\\
  0 & \text{ if } f(x)< k.
 \end{array}
 \right.
\]
Then,
\[
 \nabla \tilde{Z}_k(x) \coloneqq \left\{
 \begin{array}{ll}
  0 & \text{ if } f(x)\ge k+1,\\
  \nabla f(x) & \text{ if } k\le f(x)\le k+1,\\
  0 & \text{ if } f(x)< k.
 \end{array}
 \right.
\]
The support of $\nabla \tilde{Z}_k $ is the set
\[
B_k\coloneqq \left\{x\in B_1^2\colon A_{k+1}\le |x| \le A_k\right\}, 
\]
where
\[
 A_k = \sqrt{\frac{1}{e^{e^{k}}-1}}, \quad A_{k+1}\le A_k,\quad \text{ and } \quad\lim_{k\to\infty} A_k =0.
\]
Now,
\[
 \int_{B_1}|\nabla \tilde{Z}_k|^2 \dif x= \int_{A_{k+1}\le |x|\le A_k}  |\nabla \tilde{Z}_k|^2 \dif x \xrightarrow{k\to\infty}0,
 \]
which follows from the fact that $\nabla \tilde{Z}_k\in L^2(B_1^2)$ and that $|\{x\in B_1^2\colon A_{k+1}\le |x| \le A_k\}|$ shrinks to zero.

Thus, we obtained a sequence of functions for which 
\[
 \tilde{Z}_k\equiv 1 \text{ on } B_{A_{k+1}},\quad \tilde{Z}_k\equiv 0 \text{ outside } B_{A_k}, \quad \text{ and } \lim_{k\to\infty }\|\nabla \tilde{Z}_k\|_{L^2(B^2_1)}=0.
\]
By extending by zero we obtain a sequence $Z_k\in W^{1,2}(\R^2_+)$ with the properties
\begin{equation}\label{eq:propertiesofZk}
 {Z}_k\equiv 1 \text{ on } B_{A_{k+1}},\quad {Z}_k\equiv 0 \text{ outside } B_{A_k}, \quad \text{ and } \lim_{k\to\infty }\|\nabla {Z}_k\|_{L^2(\R^2_+)}=0.
\end{equation}
Defining now $\zeta_k\coloneqq Z_k\big\rvert_{\R}$ in the trace sense we obtain by the trace inequality, \cite{Gagliardo} 
\[
 [\zeta_k]_{W^{\frac12,2}(\R)} \aleq \|\nabla Z_k\|_{L^2(\R^2_+)} \xrightarrow{k\to\infty}0.
\]
Approximating $\{\zeta_k\}_{k\in\N}$ by smooth functions we obtain the desired sequence.
\end{proof}

\bibliographystyle{abbrv}%
\bibliography{bib}%
\end{document}